\documentclass[a4,11pt]{article}
\usepackage[dviout]{pict2e}
\usepackage{lscape}
\usepackage{graphicx}
\usepackage{framed}
\usepackage{latexsym}
\usepackage{amsthm}
\newtheorem{theorem}{Theorem}
\newtheorem{proposition}{Proposition}
\newtheorem{lemma}{Lemma}

\def\vec#1{\mbox{\boldmath $#1$}}
\usepackage{amsmath,amssymb}

\usepackage{amssymb}
\usepackage{latexsym}
\def\vec#1{\mbox{\boldmath $#1$}}
\usepackage{amsmath}

\DeclareMathOperator{\diag}{diag}

\DeclareMathOperator*{\argmin}{argmin}
\usepackage{bm}

\begin{document}

\title{Asymptotics for penalized spline estimators in quantile regression}

\author{
{\sc Takuma Yoshida}$^{1}$\\
$^{1}${\it Graduate School of Science and Engineering}\\
{\it Shimane University, Matsue 690-8504, Japan}
}

\date{\empty}
\maketitle

%
\begin{abstract} 
Quantile regression predicts the $\tau$-quantile of the conditional distribution of a response variable given the explanatory variable for $\tau\in(0,1)$.                                             
The aim of this paper is to establish the asymptotic distribution of the quantile estimator obtained by penalized spline method. 
A simulation and an exploration of real data are performed to validate our results.
\end{abstract}


\section{Introduction}

\quad Regression analysis is one of the most important tools used to investigate the relationship between a response $Y$ and a predictor $X$. 
Many major studies of regression have been concerned with the estimation of the conditional mean function of  $Y$ given a predictor $X=x$.
On the other hand, the estimation of the conditional quantile function of $Y$ given $x$ has gained momentum in recent years. 
This analysis is called quantile regression.
In quantile regression, the purpose is to estimate an unknown function $\eta_\tau(x)$ that satisfies 
$$
P(Y<\eta_\tau(x)|X=x)=\tau
$$ 
for a given $\tau\in(0,1)$. 
When $\tau=0.5$, $\eta_\tau(x)$ is the conditional median of $Y$. 
One established advantage of quantile regression as compared to mean regression is that the estimators are more robust against outliers in the response measurements.
Quantile regression models have been suggested by Koenker and Bassett (1978).
Many authors have studied quantile regression based on the parametric method, its asymptotic theories, the computational aspects and other properties, and these developments have been summarized by Koenker (2005) and Hao and Naiman (2007).
The nonparametric methods for quantile regression have also been studied extensively. 
Many authors have explored the topic in relation to kernel methods, including Fan et al. (1994), Yu and Jones (1998), Takeuchi et al. (2006), Kai et al. (2011). 
On the other hand, Hendricks and Koenker (1992) and Koenker et al. (1994) used the low-rank regression splines method and the smoothing splines method, respectively. 
Pratesi et al. (2009) and Reiss and Huang (2012) utilized the penalized spline smoothing method. 
This paper focuses on penalized splines.
Compared with unpenalized splines and smoothing splines, an advantage of the penalized spline methods is follows.
Although the smoothing spline estimator gives the predictor with fitness and smoothness, the computational cost to construct the estimator is high. 
In unpenalized regression spline methods, on the other hand, it is known that the estimator tends to have a wiggle curve, but the computational cost is lower than that of smoothing spline methods. 
The penalized spline estimator, however, gives the curve with fitness and smoothness and its computational cost is lower than that of smoothing spline methods. 
Thus, penalized splines can be considered an efficient technique.

Previous results of asymptotic studies of nonparametric quantile regressions include the following.
Fan et al. (1994) showed the asymptotic normality of the kernel estimator. 
Yu and Jones (1998) proposed a new kernel estimator and studied its asymptotic results. 
He and Shi (1994) showed the convergence rate of the unpenalized regression spline estimator. 
Portnoy (1997) discussed asymptotics for smoothing spline estimators. 
However, the asymptotics for the penalized spline estimator of quantile regression have not yet been studied.

In this paper, we show the asymptotic distribution of the penalized spline estimator for quantile regression with a low-rank $B$-spline model and the difference penalty. 
The penalized spline estimator of $\eta_\tau(x)$ for a given $\tau$ is defined as the minimizer of the convex loss function, which is the check function $\rho_\tau$ with an additional difference penalty. 
To establish the asymptotic distribution of the penalized spline estimator, we need to derive two biases (i) the model bias between the true function $\eta_\tau(x)$ and the $B$-spline model, and (ii) the bias arising from using the penalty term. 
By showing the asymptotic form of these two biases, the resulting asymptotic bias of the penalized spline estimator can be obtained. 
Finally, together with the asymptotic variance of the estimator, we show the asymptotic normality of the penalized spline quantile estimator.

This paper is organized as follows. 
In Section 2, we define the penalized spline quantile estimator for a given $\tau$. 
In terms of our estimation method, we mainly focus on the penalized iteratively reweighted least squares method. 
Section 3 provides the asymptotic bias and variance as well as the asymptotic distribution of the penalized spline quantile estimator. 
Furthermore, the related properties are described.  
In Section 4, numerical studies are conducted. 
Related discussion and issues for future research are provided in Section 5. 
Finally, proofs for the theoretical results are all given in the Appendix.

\section{Penalized spline estimator in quantile regression}

For a given dataset $\{(y_i,x_i):i=1,\cdots,n\}$, consider the conditional $100\tau\%$ quantile of response $Y_i$ given $X_i=x_i$ as 
\begin{eqnarray*}
P(Y_i<\eta_\tau(x_i)|X_i=x_i)=\tau,
\end{eqnarray*}
where $\tau\in(0,1)$ and $\eta_\tau(x_i)$ is an unknown true conditional quantile function of $Y_i$ given $X_i=x_i$. 
It is easy to show that the true function satisfies 
\begin{eqnarray*}
\eta_\tau(x)=\argmin_{a}E[\rho_\tau(Y-a(x))|X=x].
\end{eqnarray*}
Here, $\rho_\tau$ is the check function provided by Koenker and Bassett (1978), given as 
\begin{eqnarray*}
\rho_\tau(u)=u(\tau-I(u<0)),
\end{eqnarray*}
where $I(u<b)$ is the indicator function of $(-\infty,b)$. 
We want to estimate $\eta_\tau(x)$ using penalized spline methods. 
To approximate $\eta_\tau(x)$, we consider the $B$-spline model
\begin{eqnarray*}
s_\tau(x)=\sum_{k=-p+1}^K B_k^{[p]}(x)b_k(\tau),
\end{eqnarray*}
where $B_k^{[p]}(x)(k=-p+1,\cdots,K)$ are the $p$th degree $B$-spline basis functions defined recursively as 
\begin{eqnarray*}
B_k^{[0]}(x)&=&
\left\{
\begin{array}{cc}
1,& \kappa_{k-1}<x\leq \kappa_k,\\
0,& {\rm otherwise},
\end{array}
\right. \\
B_k^{[p]}(x)&=&\frac{x-\kappa_{k-1}}{\kappa_{k+p-1}-\kappa_{k-1}}B_k^{[p-1]}(x)+\frac{\kappa_{k+p}-x}{\kappa_{k+p}-\kappa_{k}}B_{k+1}^{[p-1]}(x),
\end{eqnarray*}
where $\kappa_k(k=-p+1,\cdots,K+p)$ are knots 
and $b_{k}(\tau)(k=-p+1,\cdots,K)$ are unknown parameters. 
We denote $B_k^{[p]}(x)$ as $B_k(x)$ unless the degrees of $B$-splines are specified. 
Details and many properties of the B-spline function are clarified by de Boor (2001).
The estimator of $\vec{b}(\tau)=(b_{-p+1}(\tau)\ \cdots\ b_{K_n}(\tau))^T$ is defined as 
\begin{eqnarray}
\hat{\vec{b}}(\tau)
&=&
(\hat{b}_{-p+1}(\tau)\ \cdots\ \hat{b}_{K_n}(\tau))^T\nonumber\\
&=&
\underset{\vec{b}(\tau)}{\argmin}\left\{ \sum_{i=1}^n \rho_\tau(y_i-\vec{B}(x_i)^T\vec{b}(\tau))+\frac{\lambda_\tau}{2}\vec{b}(\tau)^T D_m^T D_m \vec{b}(\tau)\right\}, \label{plsc}
\end{eqnarray}
where $\vec{B}(x_j)=(B_{-p+1}(x_j)\ \cdots\ B_{K_n}(x_j))^T$, $\lambda_\tau(>0)$ is the smoothing parameter and $(K_n+p-m)\times(K_n+p)$th matrix $D_m$ is the $m$th difference matrix, which is defined as $D_m=(d^{(m)}_{ij})_{ij}$, where $d^{(m)}_{ij}=(-1)^{|i-j|}{}_{m}C_{|i-j|}$ for $i\leq j\leq m+1$, and 0 for otherwise.
It is well known that the difference penalty in (\ref{plsc}) is very useful in mean regression and can be regarded as the controller of the smoothness of $s_\tau(x)$ because we can interpret $\vec{b}(\tau)^T D_m^T D_m \vec{b}(\tau)\approx K_n^{2m-1}\int_0^1 \{s_\tau^{(m)}(x)\}^2dx$ (see, Eilers and Marx (1996)). 
Although Reiss and Huang (2012) used the penalty $\int_0^1 \{s_\tau^{(m)}(x)\}^2dx$, this penalty contains an integral and hence the computational difficulty for the resulting estimator grows. 
Therefore, this paper proposes using $\vec{b}(\tau)^T D_m^T D_m \vec{b}(\tau)$ as the penalty.
In fact, $\hat{\vec{b}}(\tau)$ is obtained via linear-programming methods, such as simplex methods or interior points methods (see Koenker and Park (1996), Koenker (2005)). 
On the other hand, it is known that the iteratively reweighted least squares (IRLS) method is a useful in nonparametric quantile regression.  
The penalized spline estimator obtained via IRLS was also studied and detailed by Reiss and Huang (2012).

Since IRLS is important for obtaining the estimator, we now provide the complete algorithm. 
For a given $\lambda_\tau$, the $k$-steps iterated estimator $\hat{\vec{b}}^{(k)}(\tau)$ is defined as follows: 
\begin{eqnarray*}
\hat{\vec{b}}^{(k)}(\tau)=(Z^T W^{(k-1)}Z+\lambda_\tau D_m^T D_m)^{-1}Z^T W^{(k-1)}\vec{y},
\end{eqnarray*}
where $\vec{y}=(y_1\ \cdots\ y_n)^T$, $Z=(B_{-j+p}(x_{i}))_{ij}$, $W^{(k)}=\diag[w_1^{(k)}\ \cdots\ w_n^{(k)}]$ and 
\begin{eqnarray*}
w_i^{(k)}=
\left\{
\begin{array}{ll}
\displaystyle\frac{\tau-I(y_i-\vec{B}(x_i)^T\hat{\vec{b}}^{(k-1)}(\tau)<0)}{2(y_i-\vec{B}(x_i)^T\hat{\vec{b}}^{(k-1)}(\tau))} 
&
|y_i-\vec{B}(x_i)^T\hat{\vec{b}}^{(k-1)}(\tau))|>\alpha,\\
\displaystyle \frac{\tau(y_i-\vec{B}(x_i)^T\hat{\vec{b}}^{(k-1)}(\tau))}{\alpha}
&
0\leq y_i-\vec{B}(x_i)^T\hat{\vec{b}}^{(k-1)}(\tau))\leq \alpha,\\
\displaystyle\frac{(1-\tau)(y_i-\vec{B}(x_i)^T\hat{\vec{b}}^{(k-1)}(\tau))}{\alpha}
&
-\alpha\leq y_i-\vec{B}(x_i)^T\hat{\vec{b}}^{(k-1)}(\tau)\leq 0,
\end{array}
\right.
\end{eqnarray*}
for small $\alpha>0$ and the initial $W^{(0)}$.
As $k\rightarrow \infty$, it can be shown that $\lim_{k\rightarrow \infty}\hat{\vec{b}}^{(k)}(\tau)$ is approximately equivalent to the minimizer of (\ref{plsc}).
By using $\hat{\vec{b}}(\tau)$, the penalized spline estimator of $\eta_\tau(x)$ is defined as 
\begin{eqnarray*}
\hat{\eta}_\tau(x)=\sum_{k=-p+1}^K B_k^{[p]}(x)\hat{b}_k(\tau)={B}(x)^T\hat{\vec{b}}(\tau).
\end{eqnarray*}


\section{Asymptotic theory}

In this section, we show the asymptotic property of $\hat{\eta}_\tau(x)$. 
Then, we assume that the number of knots $K$ and smoothing parameter $\lambda_\tau$ are dependent on $n$, and we write $K_n$ and $\lambda_{\tau,n}$, respectively.
For simplicity, we write $\lambda_{n}=\lambda_{\tau,n}$.
We give some assumptions regarding the asymptotics of the penalized spline quantile estimator.
\\

{\noindent \bf Assumptions}
\begin{enumerate}
\item The explanatory $X$ is distributed as $Q(x)$ on $[0,1]$. 
\item The knots for the $B$-spline basis are equidistantly located as $\kappa_k=k/K_n(k=-p+1,\cdots,K_n+p)$ and the number of knots satisfies $K_n=o(n^{1/2})$. 
\item There exists $\gamma\geq 0$ such that $E[|(\tau-I(Y<\eta_\tau(x)))|^{2+\gamma}|X=x]<\infty$. 
\item The order of the difference matrix is $m\leq p+1$. 
\item The smoothing parameter $\lambda_{n}$ is a positive sequence such that $\lambda_n^{-1}$ is larger than the maximum eigenvalue of $G(\tau)^{-1/2}D_m^TD_mG(\tau)^{-1/2}$. 
\end{enumerate}

To describe the asymptotic form of $\hat{\eta}_\tau(x)$, we introduce the following symbols and notations. 
Define the $(K_n+p)$th square matrix $G=(G_{ij})_{ij}$ by 
\begin{eqnarray*}
G_{ij}=\int_0^1 B_{-p+i}(u)B_{-p+j}(u)dQ(u)
\end{eqnarray*}
and the $(K_n+p)$th square matrix $G(\tau)$ as having the $(i,j)$-component
\begin{eqnarray*}
G_{ij}(\tau)=\int_0^1 f(\eta_\tau(u)|u)B_{-p+i}(u)B_{-p+j}(u)dQ(u),
\end{eqnarray*}
where $f(y|x)$ is the conditional density function of $Y$ given $X=x$.

Let $\vec{b}^*(\tau)$ be a best $L_{\infty}$ approximation to the true function $\eta_{\tau}(x)$, which satisfies
\begin{eqnarray}
\sup_{x\in(0,1)}\left|\eta_\tau(x)+b^a_\tau(x)-\vec{B}(x)^\prime \vec{b}^*(\tau)\right|=o(K_n^{-(p+1)}), \label{app}
\end{eqnarray}
where 
\begin{eqnarray*}
b^a_\tau(x)=-\frac{\eta_\tau^{(p+1)}(x)}{K_n^{p+1} (p+1)!}\sum_{k=1}^{K_n}I(\kappa_{k-1}\leq x<\kappa_k){\rm Br}_{p+1}\left(\frac{x-\kappa_{k-1}}{K_n^{-1}}\right),
\end{eqnarray*}
$I(a<x<b)$ is the indicator function of an interval $(a,b)$ and ${\rm Br}_p(x)$ is the $p$th Bernoulli polynomial(see Zhou et al. (1998)). 
Next, we use $\eta_{\tau}^{*}(x)=\vec{B}(x)^T\vec{b}^*(\tau)$.

The penalized spline quantile estimator can be decomposed as 
\begin{eqnarray*}
\hat{\eta}_\tau(x)-\eta_\tau(x)=\hat{\eta}_\tau(x)-\eta_{\tau}^{*}(x) +\eta_{\tau}^{*}(x)-\eta_\tau(x)=\hat{\eta}_\tau(x)-\eta_{\tau}^{*}(x)+b^a_\tau(x)+o(K_n^{-(p+1)}).
\end{eqnarray*} 
We investigate the asymptotic distribution of $\hat{\eta}_\tau(x)-\eta_{\tau}^{*}(x)$ in the following Proposition.

\begin{proposition}\label{para}
Let $\eta_\tau(\cdot)\in C^{p+1}$. 
Furthermore suppose $K_n=O(n^{1/(2p+3)})$ and $\lambda_{n}=O(n^{\nu}), \nu\leq (p+m+1)/(2p+3)$. 
Then under the Assumptions, for $x\in(0,1)$, as $n\rightarrow \infty$,
\begin{eqnarray*}
\sqrt{\frac{n}{K_n}} \{\hat{\eta}_\tau(x)-\eta_{\tau}^{*}(x)-b^\lambda_\tau(x)\}\stackrel{D}{\longrightarrow} N(0,\Phi_\tau(x)),
\end{eqnarray*}
where 
\begin{eqnarray*}
b^\lambda_\tau(x)&=&-\frac{\lambda_{n}}{n}\vec{B}(x)^T (G(\tau)+(\lambda_{n}/n)D_m^T D_m)^{-1}D_m^T D_m\vec{b}^*(\tau)=O(n^{-(p+1)/(2p+3)}),\\
\Phi_\tau(x)&=&\lim_{n\rightarrow \infty}\frac{\tau(1-\tau)}{K_n}\vec{B}(x)^T (G(\tau)+(\lambda_{n}/n) D_m^T D_m)^{-1}G(G(\tau)+(\lambda_{n}/n) D_m^T D_m)^{-1}\vec{B}(x).
\end{eqnarray*}
\end{proposition}

The following Theorem, which is the main result in this paper, can be obtained straightforwardly from Proposition \ref{para}. 

\begin{theorem}\label{clt}
Under the same assumptions as Proposition \ref{para}, for $x\in(0,1)$, as $n\rightarrow \infty$,
\begin{eqnarray*}
\sqrt{\frac{n}{K_n}} \{\hat{\eta}_\tau(x)-\eta_\tau(x)-b^a_\tau(x)-b^\lambda_\tau(x)\}\stackrel{D}{\longrightarrow} N(0,\Phi_\tau(x)),
\end{eqnarray*}
where $b^\lambda_\tau(x)$ and $\Phi_\tau(x)$ are those given in Proposition \ref{para}.
\end{theorem}

\noindent{\bf Remark 1}
\quad Under the same assumption as Theorem \ref{clt}, the rate of convergence of the mean squared error(MSE) of $\hat{\eta}_\tau(x)$ becomes
\begin{eqnarray*}
E\left[\{\hat{\eta}_\tau(x)-\eta_\tau(x)\}^2\right]=O(n^{-(2p+2)/(2p+3)}).
\end{eqnarray*}
This rate is the same as that of the penalized spline estimator in mean regression (see, Kauermann et al. (2009)). 

\vspace{5mm}

\noindent{\bf Remark 2}
\quad  
For the unpenalized regression spline quantile estimator, its asymptotic normality is obtained through Theorem \ref{clt} with $\lambda_{n}=0$.

\vspace{5mm}

\noindent{\bf Remark 3}
\quad When the true quantile function has a polynomial form $\eta_\tau(x)=a_0+a_1x+\cdots+a_qx^q (q\leq p)$, $\eta_\tau(x)=\eta_\tau^*(x)$ is satisfied since the $q$th polynomial model can be expressed as the linear combination of the $p$th $B$-spline bases $\{B^{[p]}_{k}:k=-p+1,\cdots,K_n\}$(see de Boor (2001)). 
Therefore, in this case, the model bias becomes 0, indicating that the regression spline quantile estimator is unbiased. 
We can definitely show that $E[\psi_\tau(U_i)|\vec{X}_n]=0$ in the proof of Lemma \ref{Lyapnov}. 

\vspace{5mm}

\noindent{\bf Remark 4}
\quad Let $\varepsilon_i(i=1,\cdots,n)$ be independently and identically distributed as the density $f_\varepsilon(\varepsilon)$ and assume that $X_i$ and $\varepsilon_i$ are independent.
Consider the data $\{(y_i,x_i):i=1,\cdots,n\}$ with $Y_i=\eta(x_i)+\varepsilon_i$. 
Then the conditional $100\tau\%$ quantile of $Y_i$ given $X_i=x_i$ can be written as $\eta_\tau(x_i)=\eta(x_i)+F_\varepsilon^{-1}(\tau)$, where $F_\varepsilon^{-1}(\tau)$ is the $100\tau\%$ quantile of $\varepsilon_i$. 
For any $\tau\in(0,1)$, $\eta_\tau^{(p+1)}(x)=\eta^{(p+1)}(x)$, with which $b_\tau^a(x)$ is unchanged by $\tau$. 
Next, we obtain $G(\tau)=f_\varepsilon(F_\varepsilon^{-1}(\tau)) G$ since $f(\eta_\tau(x)|x)$ is equal to
$f_\varepsilon(F_\varepsilon^{-1}(\tau))$.
Furthermore, $\vec{b}^*(\tau)$ can be written as $\vec{b}^*(\tau)=\vec{b}^*+F_\varepsilon^{-1}(\tau)\vec{1}$, where $\vec{b}^*$ is the best $L_\infty$ approximation of $\eta(x)$ defined in the same manner as $\vec{b}^*(\tau)$ and $\vec{1}$ is a $(K_n+p)$ vector with all components equal to 1. 
Since all components of $D_m\vec{1}$ are vanishing, for $\tau\in(0,1)$, we have
\begin{eqnarray*}
b^\lambda_\tau(x)=-\frac{\lambda_{n}}{nf_\varepsilon(F_\varepsilon^{-1}(\tau))}\vec{B}(x)^T \left(G+\frac{\lambda_{n}}{nf_\varepsilon(F_\varepsilon^{-1}(\tau))} D_m^T D_m\right)^{-1}D_m^T D_m\vec{b}^*.
\end{eqnarray*}
The asymptotic variance of $\hat{\eta}_\tau(x)$ can be written as 
\begin{eqnarray*}
 \Phi_\tau(x)=\lim_{n\rightarrow \infty}\alpha_n(\tau)\vec{B}(x)^T \left(G+\frac{\lambda_{n}}{nf_\varepsilon(F_\varepsilon^{-1}(\tau))} D_m^T D_m\right)^{-1}G\left(G+\frac{\lambda_{n}}{nf_\varepsilon(F_\varepsilon^{-1}(\tau))} D_m^T D_m\right)^{-1}\vec{B}(x),
\end{eqnarray*}
where 
\begin{eqnarray*}
\alpha_n(\tau)=\frac{\tau(1-\tau)}{\{f_\varepsilon(F_\varepsilon^{-1}(\tau))\}^2K_n}.
\end{eqnarray*}
When the sample size is sufficiently large under the same assumptions as Theorem \ref{clt} and $m<p+1$, the influences of $\tau$ on $b^\lambda_\tau(x)$ and $\Phi_\tau(x)$ appear only as $1/f_\varepsilon(F_\varepsilon^{-1}(\tau))$ and $\tau(1-\tau)/\{f_\varepsilon(F_\varepsilon^{-1}(\tau))\}^2$, respectively. 
In general, if the density of $\varepsilon_i$ is symmetrical at $\varepsilon=0$, the asymptotic bias and variance of $\hat{\eta}_\tau(x)$ are small at $\tau=0.5$.  
Figure \ref{exam} shows $1/f_\varepsilon(F_\varepsilon^{-1}(\tau))$ and $\tau(1-\tau)/\{f_\varepsilon(F_\varepsilon^{-1}(\tau))\}^2$ with normal and Cauchy distributions.

We observe that $b^\lambda_\tau(x)$ and $\Phi_\tau(x)$ are smallest at $\tau=0.5$. 
For $\Phi_\tau(x)$ near $\tau=$ or $\tau=1$, the effect of $\tau$ becomes small.

\begin{figure}
\begin{center}
\includegraphics[width=50mm,height=40mm]{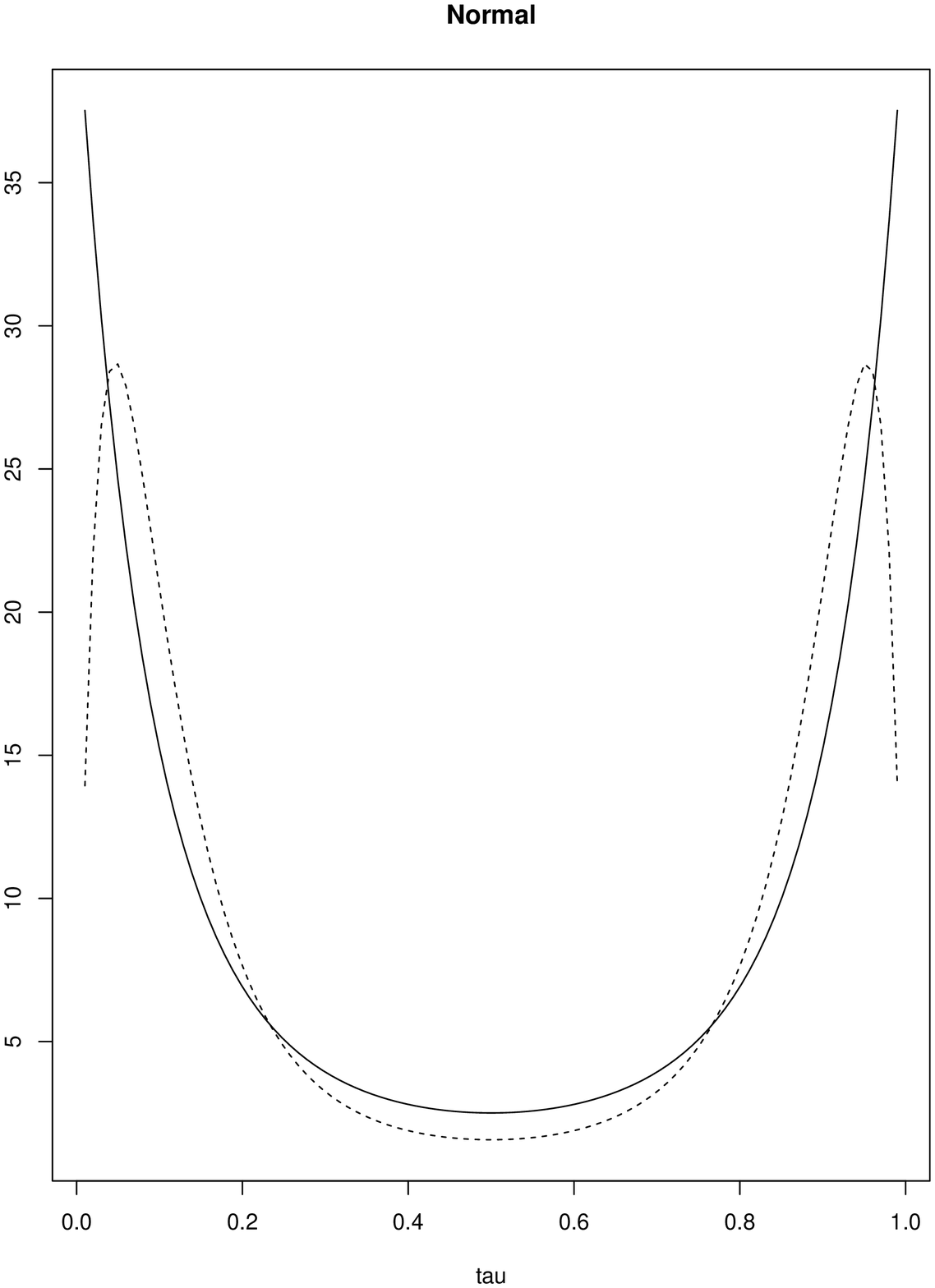}
\includegraphics[width=50mm,height=40mm]{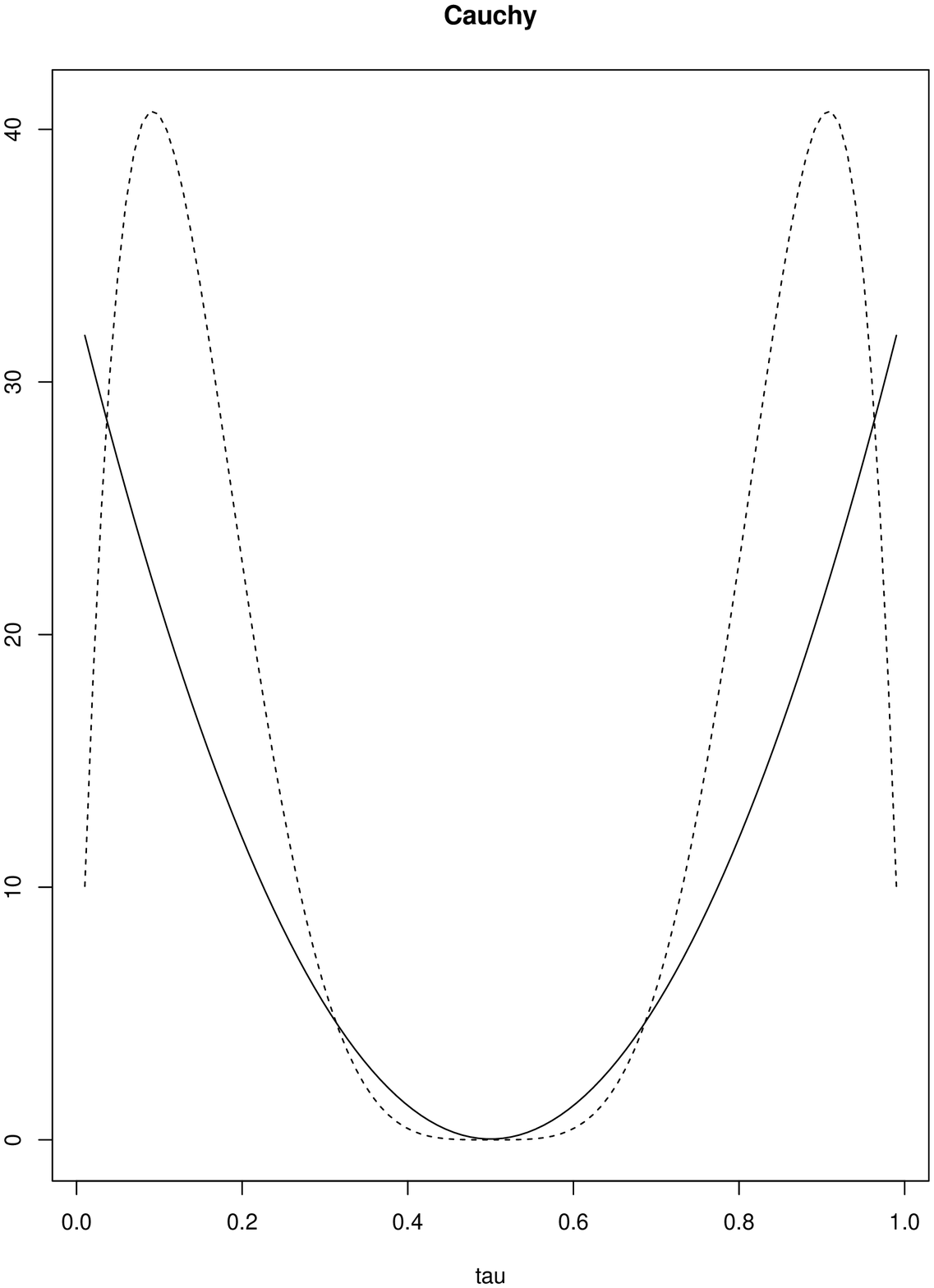}
\end{center}
\caption{Plots for $1/f_\varepsilon(F_\varepsilon^{-1}(\tau))$(solid) and $\tau(1-\tau)/\{f_\varepsilon(F_\varepsilon^{-1}(\tau))\}^2$(dashed). 
The left panel shows the standard normal distribution and the right panel shows the Cauchy distribution with location 0 and scale 0.01.\label{exam}}
\end{figure}

\vspace{5mm}

\noindent{\bf Remark 5} 
\quad Claeskens et al. (2009) studied the asymptotics of penalized spline estimators in mean regression, with the estimator $\hat{\eta}(x)=\vec{B}(x)^T\hat{\vec{b}}$, where $\hat{\vec{b}}$ is the minimizer of 
\begin{eqnarray}
(\vec{y}-Z\vec{b})^\prime(\vec{y}-Z\vec{b})+\mu_n\int_0^1\{s^{(m)}(x)\}^2dx.  \label{npls}
\end{eqnarray}
Here, $s(x)=\vec{B}(x)^T \vec{b}, \vec{b}\in\mathbb{R}^{K_n+p}$ and $\mu_n$ is the smoothing parameter. 
They developed the asymptotics for $\hat{\eta}(x)$  under two scenarios: (a) $K_q=K_q(n,K_n,\mu_n)$, which as given in their paper is less than 1, or (b) $K_q\geq 1$. 
Assumption 5 of this paper is equal to the condition $K_q<1$.  
Together with the approximation property that $\lambda_{n}\vec{b}(\tau)^T D_m^T D_m \vec{b}(\tau)\approx \mu_nK_n^{2m-1}\int_0^1 \{s_\tau^{(m)}(x)\}^2dx$, 
the results of this paper can be regarded as the quantile regression version of Theorem 2 (a) of Claeskens et al. (2009).

\vspace{5mm}

\noindent{\bf Remark 6} 
\quad  To construct the penalized spline estimator of $\eta_\tau(x)$, we can also use the truncated spline
$c_\tau(x)=\vec{C}(x)^T\vec{\theta}(\tau)$ as an approximation to $\eta_\tau(x)$, where $\vec{C}(x)=[1\ x\ \cdots\ x^p\ (x-\kappa_1)_+^p\ \cdots\ (x-\kappa_{K_n-1})_+^p]$, $(x)_+=\max\{x,0\}$, and $\vec{\theta}(\tau)\in\mathbb{R}^{K_n+p}$ is an unknown parameter vector. 
Pratesi et al. (2009) obtained the estimator $\tilde{\eta}_{\tau}(x)=\vec{C}(x)^T\tilde{\vec{\theta}}(\tau)$, where $\tilde{\vec{\theta}}(\tau)$ is the minimizer of 
\begin{eqnarray}
\sum_{i=1}^n \rho_\tau(y_i-c_\tau(x_i))+\mu_n\vec{\theta}(\tau)^T\Theta\vec{\theta}(\tau),\label{pentr}
\end{eqnarray}
where $\mu_n$ is the smoothing parameter and $\Theta=\diag[O_{p+1}\ I_{K_n-1}]$.
By the equivalence property between the $B$-spline model and truncated model, there exists a $(K_n+p)$th square and nonsingular matrix $L$ such that $\vec{B}(x)=L\vec{C}(x)$. 
Therefore $c_\tau(x)$ can be written as 
\begin{eqnarray*}
c_\tau(x)=\vec{C}(x)^T\vec{\theta}=\vec{B}(x)^TL^{-1}\vec{\theta}(\tau)=\vec{B}(x)^T\vec{b}(\tau),
\end{eqnarray*}
where $\vec{b}(\tau)=L^{-1}\vec{\theta}(\tau)$. 
Furthermore, the penalty term in (\ref{pentr}) satisfies from Claeskens et al. (2009)
$$
\vec{\theta}(\tau)^T\Theta\vec{\theta}(\tau)=K_n^{2p}\vec{b}(\tau)^TD_{p+1}^TD_{p+1}\vec{b}(\tau)
$$
The asymptotic distribution of $\tilde{\eta}_{\tau}(x)=\vec{C}(x)^T\tilde{\vec{\theta}}(\tau)$ can be obtained by showing that of $\vec{B}(x)^T\tilde{\vec{b}}$, where $\tilde{\vec{b}}$ is the minimizer of 
\begin{eqnarray*}
\sum_{i=1}^n \rho_\tau\left(y_i-\vec{B}(x_i)^T\vec{b}(\tau)\right)+\mu_nK_n^{2p}\vec{b}(\tau)^TD_{p+1}^TD_{p+1}\vec{b}(\tau).
\end{eqnarray*}
Then, the asymptotic distribution of $\tilde{\eta}_{\tau}(x)$ can be obtained using Theorem \ref{clt} under $m=p+1$ and $\lambda_n=\mu_nK_n^{2p}$.
Thus, we obtain the asymptotic distribution of the penalized truncated spline quantile estimator.

\vspace{5mm}

\noindent{\bf Remark 7} 
\quad Under some weakly condition, the local $p$th polynomial quantile estimator $\tilde{\eta}_\tau(x)$ has an asymptotic order 
$$
E[\{\tilde{\eta}_\tau(x)-\eta_\tau(x)\}^2]=O(n^{-2(p+1)/(2p+3)})
$$
(see Fan et al. (1994) and Ghouch and Genton (2009)) and, hence, it can be said that the rate of convergence of the $p$th $B$-spline quantile estimator and the local $p$th polynomial quantile estimator are the same. 
We note the bias of these estimators with $p=1$. 
From Fan et al. (1994), the asymptotic bias of the local linear quantile estimator is 
\begin{eqnarray*}
b_\tau^{\ell}(x)=-\frac{h_n^2\eta_\tau^{(2)}(x)}{2}\int_{\mathbb{R}}z^2K(z)dz,
\end{eqnarray*}
where $K(z)$ is the second order kernel function and $h_n$ is the bandwidth. 
If $K_n^{-1}$ is equal to $h_n$, then the difference between $b_\tau^a(x)$ and $b_\tau^\ell(x)$ is only 
\begin{eqnarray}
\sum_{j=1}^{K_n}I(\kappa_{j-1}\leq x<\kappa_j){\rm Br}_{2}\left(\frac{x-\kappa_{j-1}}{K_n^{-1}}\right)\ \ {\rm and}\ \  \int_{\mathbb{R}} z^{2}K(z)dz. \label{const}
\end{eqnarray} 
It is easy to show that ${\rm Br}_2(x)=x^2-x+1/6<1/5$ for $x\in[0,1]$, while we have $\int_{\mathbb{R}} z^2K_G(z)dz=1$ for the Gaussian kernel $K_{G}(z)$ and $\int_{\mathbb{R}} z^2K_E(z)dz=1/5$ for the Epanechnikov kernel $K_E(z)$. 
Therefore the bias of the regression spline estimator is smaller than that of the local linear estimator in this situation.


\section{Numerical study}

\subsection{Simulation}

In this section, we show numerical simulation to confirm the performance as well as the asymptotic normality of the penalized spline quantile estimator claimed in Theorem \ref{clt}. 
The explanatory $x_{i}$ is generated from a uniform distribution on the interval $[0,1]$. 
The response $Y_i$ is created by $Y_i=\eta(x_i)+\varepsilon_i$, where $\eta(x)=\sin(2\pi x)$. 
The errors $\varepsilon_i$'s are independently distributed via (i) a normal distribution with mean 0 and variance $(0.1)^2$, (ii) an exponential distribution with mean 2 and (iii) a Cauchy distribution with location 0 and scale 0.01. 
In this simulation, to obtain the penalized spline quantile estimator, we use $(p,m)=(3,2)$ and $(K_n,\lambda_n)$ is given via the generalized approximate cross-validation (GACV) discussed by Yuan (2006).
For comparison, we construct the unpenalized regression spline quantile estimator with linear spline bases($p=1$) and the local linear quantile estimator. 
The penalized spline estimator, regression spline estimator, and local linear estimator are denoted as P-cubic, R-linear and L-linear, respectively.
The number of knots of R-linear and the bandwidth of L-linear are given by GACV.

Let 
\begin{eqnarray*}
{\rm MSE}_j=\frac{1}{R}\sum_{r=1}^{R}\{\hat{\eta}_{\tau,r}(z_j)-\eta_\tau(z_j)\}^2,\ \ {\rm MISE}=100^{-1}\sum_{j=1}^{100}{\rm MSE}_j,
\end{eqnarray*}
where $z_j=j/J, J=100$ and $\hat{\eta}_{\tau,r}(z_j)$ is the estimator for the $r$th repetition. 
For $\tau=0.01,0.1,0.25$ and 0.5, we calculate the mean integrated squared error (MISE).
We then use sample sizes $n=100$ and 1000 and the number of repetitions $R=1000$. 

Next, from P-cubic, we calculate 
\begin{eqnarray*}
U_{\tau,r}(x)=\frac{\hat{\eta}_{\tau,r}(x)-\eta_\tau(x)}{\hat{\Phi}_{\tau,r}(x)}, \ \ r=1,\cdots,R,
\end{eqnarray*} 
where 
\begin{eqnarray*}
\hat{\Phi}_{\tau,r}(x)=\tau(1-\tau)\vec{B}(x)^T(Z^T\hat{R}_rZ+\lambda_{n}D_m^T D_m)^{-1}Z^TZ(Z^T\hat{R}_rZ+\lambda_{n}D_m^T D_m)^{-1}\vec{B}(x),
\end{eqnarray*}
$\hat{R}_r=\diag[\hat{f}_r(\hat{\eta}_{\tau,r}(x_i)|x_i)]$ and $\hat{f}_r(y|x)$ is the conditional kernel density estimate given $X=x$.
Then we construct the density estimate of $U_{\tau}\equiv\{U_{\tau,1}(x),\cdots,U_{\tau,R}(x)\}$ at $x=0.5$ and compare with the density of $N(0,1)$. 
To obtain $\hat{f}_r(y|x)$ and $U_{\tau}$, the normal kernel and the bandwidth discussed by Sheather and Jones (1991) are utilized. 

\begin{table}
\begin{center}
\caption{Results of MISE for $n=100$ and $n=1000$. All entries for MISE are $10^3$ times their actual values.}
\scalebox{0.9}[0.9]{
\begin{tabular}{c|ccc|ccc|ccc}
\hline
$n=100$&\multicolumn{3}{|c|}{Normal}&\multicolumn{3}{|c|}{Exponential}&\multicolumn{3}{|c}{Cauchy}\\
\hline
$\tau$&P-cubic&R-linear&L-linear&P-cubic&R-linear&L-linear&P-cubic&R-linear&L-linear\\
\hline
0.01
&11.89&20.16&20.81
&5.21&11.15&11.78
&4704.18&6667.28&4122.68
\\
0.1
&3.78&4.55&5.03
&6.26&9.85&12.50
&206.43&340.47&289.05
\\
0.25
&3.23& 3.27&3.99
&10.72&16.68&13.42
&18.46&42.15&86.46
\\
0.5
&2.87&3.34&3.60
&20.26&31.66&27.92
&18.88& 35.22&43.33
\\
\hline
\hline
$n=1000$&\multicolumn{3}{|c|}{Normal}&\multicolumn{3}{|c|}{Exponential}&\multicolumn{3}{|c}{Cauchy}\\
\hline
$\tau$&P-cubic&R-linear&L-linear&P-cubic&R-linear&L-linear&P-cubic&R-linear&L-linear\\
\hline
0.01
&1.23&1.77&1.94
&0.22&0.31&0.31
&160.52&910.09&1178.74
\\
0.1
&0.46&1.45&0.67
&1.08&1.30&1.08
&19.53&24.99&53.93
\\
0.25
&0.44&1.84&0.52
&2.16&2.77&1.91
&2.08&7.08&2.17
\\
0.5
&0.12&0.34&0.27
&3.66&5.60&3.19
&0.20&4.81&1.38
\\
\hline
\end{tabular}
}
\end{center}
\end{table}

Table 1 shows the MISE for $\tau=0.01, 0.1, 0.25$ and 0.5. 
For P-cubic with normal error, the performance of the quantile estimator is good even if $\tau=0.01$.
It is well known that the Cauchy distribution is a pathological distribution. 
However, the MISE of P-cubic with the Cauchy distribution is sufficiently small, indicating that the quantile estimator is robust. 
For the boundary $\tau$, on the other hand, the MISE of the estimators is worse than that with interior $\tau$. 
For the normal and Cauchy models, the median estimator has better behavior than those with $\tau=0.01,0.1$ and 0.25. 
On the other hand, for the exponential model, the median estimator has a larger MISE than $\hat{\eta}_{\tau}(x)$ with other values of $\tau$. 
The reason for this is that the density $f(\varepsilon)$ of exponential distribution is monotonically decreasing and its peak is at $\varepsilon=0$, which leads to many responses $Y_i$'s being dropped near $\eta(x_i)+F_\varepsilon^{-1}(\tau)$ with small $\tau$.
We note the performance of the penalized spline estimator for $\tau>0.5$. 
When a normal or Cauchy error is used, it appears that the MISE of $\hat{\eta}_{\tau}(x)$ and that of $\hat{\eta}_{1-\tau}(x)$ become similar since $Y_i|x_i$ has a symmetrical density function at $\eta(x_i)$. 
For an exponential error, the closer $\tau$ is to 0, the smaller the MISE of $\hat{\eta}_{\tau}(x)$ will become. 
Overall, P-cubic has better behavior than R-linear and L-linear. 
However, for the exponential distribution and $n=1000$, the MISE of L-linear is slightly smaller than that of P-cubic.
Additionally, the performance of L-linear is slightly superior to that of R-linear. 
This indicates that the variance of L-linear is less than that of R-linear (see Remark 7).

In Figure \ref{simu}, the density estimate of $U_{\tau}$ for $\tau=0.1$ and 0.5 and the density of $N(0,1)$ for each error are illustrated. 
In all errors, we can see that the density estimate of $U_{0.5}(x)$ becomes close to $N(0,1)$ as $n$ increases. 
For a normal distribution with $n=1000$, the density estimate $U_{0.5}$ and $N(0,1)$ are similar. 
In both errors, we see that the speed of convergence of $U_{0.5}$ is faster than that of $U_{0.1}$. 

\vspace{5mm} 

\noindent{\bf Remark 8}
\ We have confirmed the behavior of the penalized splines with $p=1$ (P-linear) and the regression splines with $p=3$ (R-cubic) though this is not shown in this
paper for reasons of space. 
The MISE of P-linear and R-cubic are similar to the P-cubic and R-linear, respectively.
For spline smoothing, it is generally known that the pair of the ^^ cubic' spline and the second difference penalty are particularly useful in data analysis. 
Therefore we mainly focused on $(p,m)=(3,2)$ in this simulation.

\begin{figure}
\begin{center}
\includegraphics[width=75mm,height=50mm]{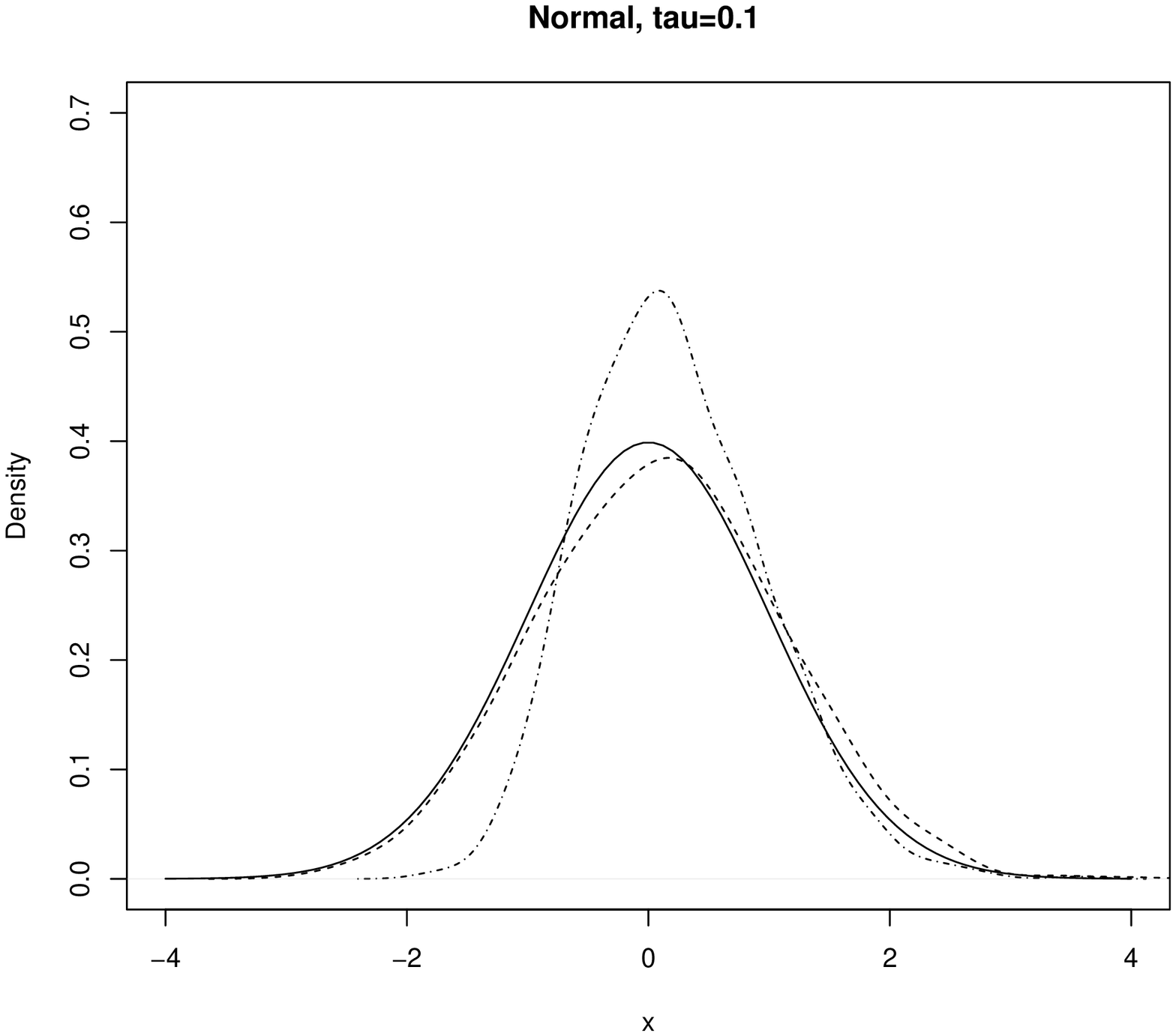}
\includegraphics[width=75mm,height=50mm]{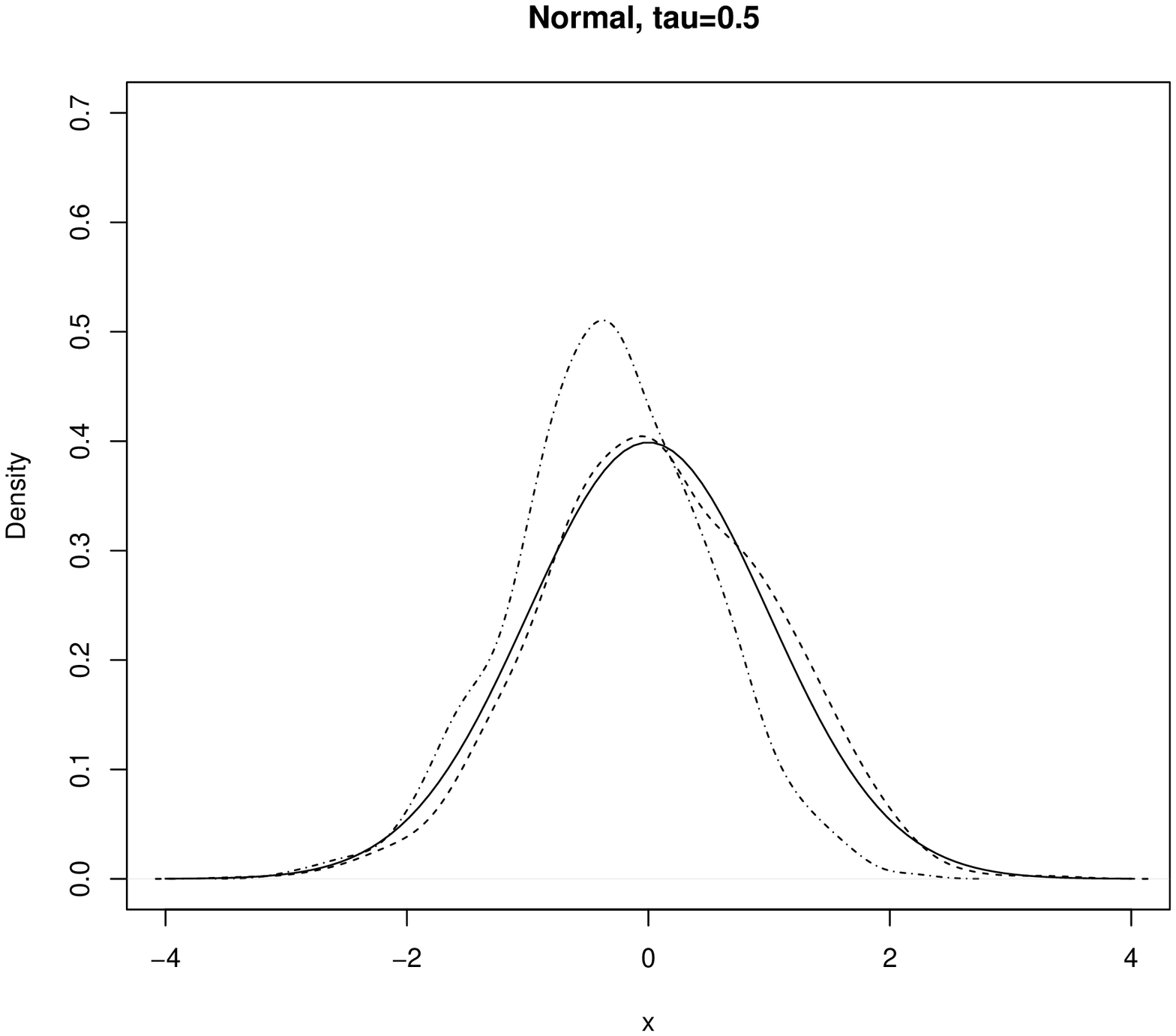}\\
\includegraphics[width=75mm,height=50mm]{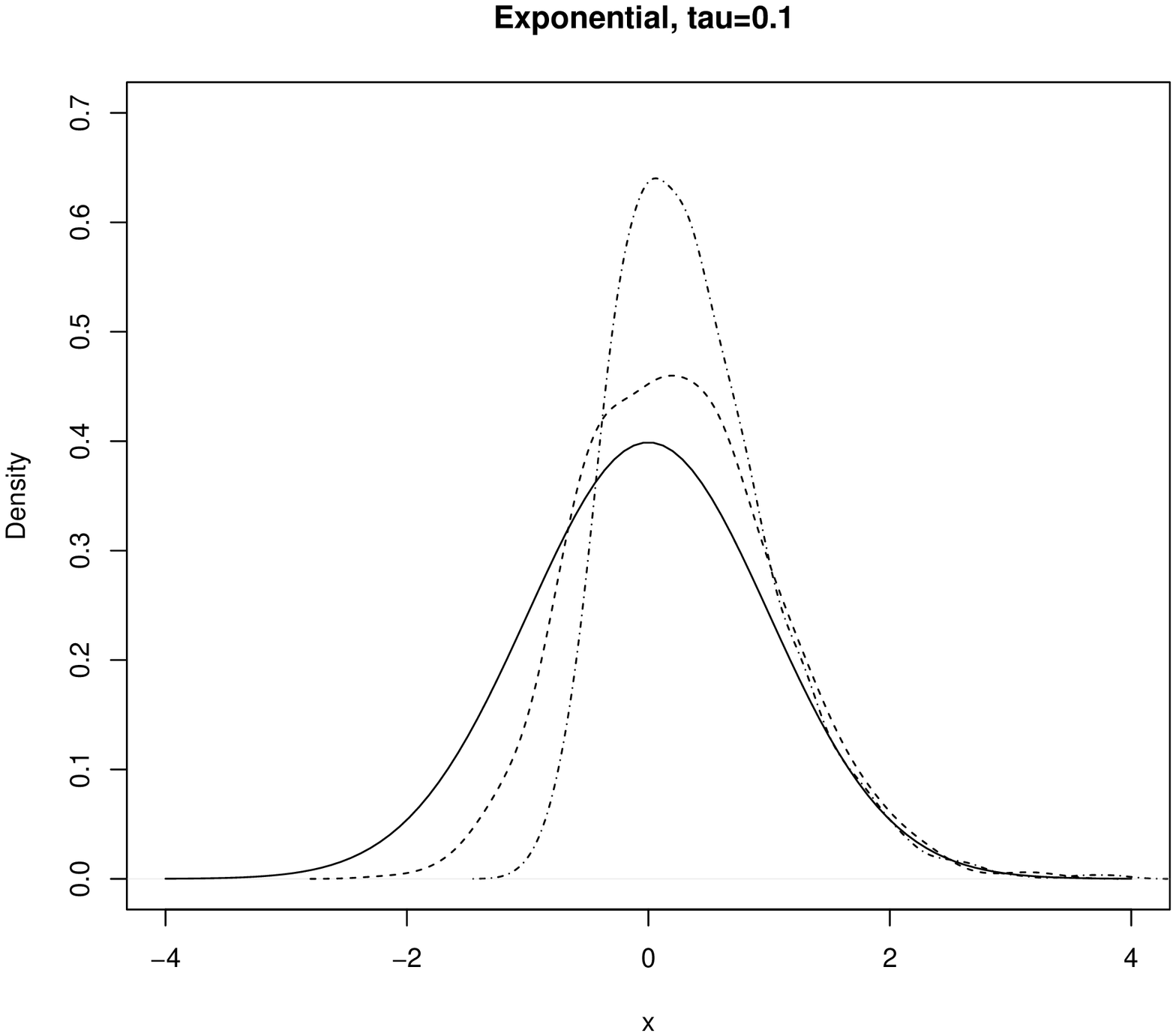}
\includegraphics[width=75mm,height=50mm]{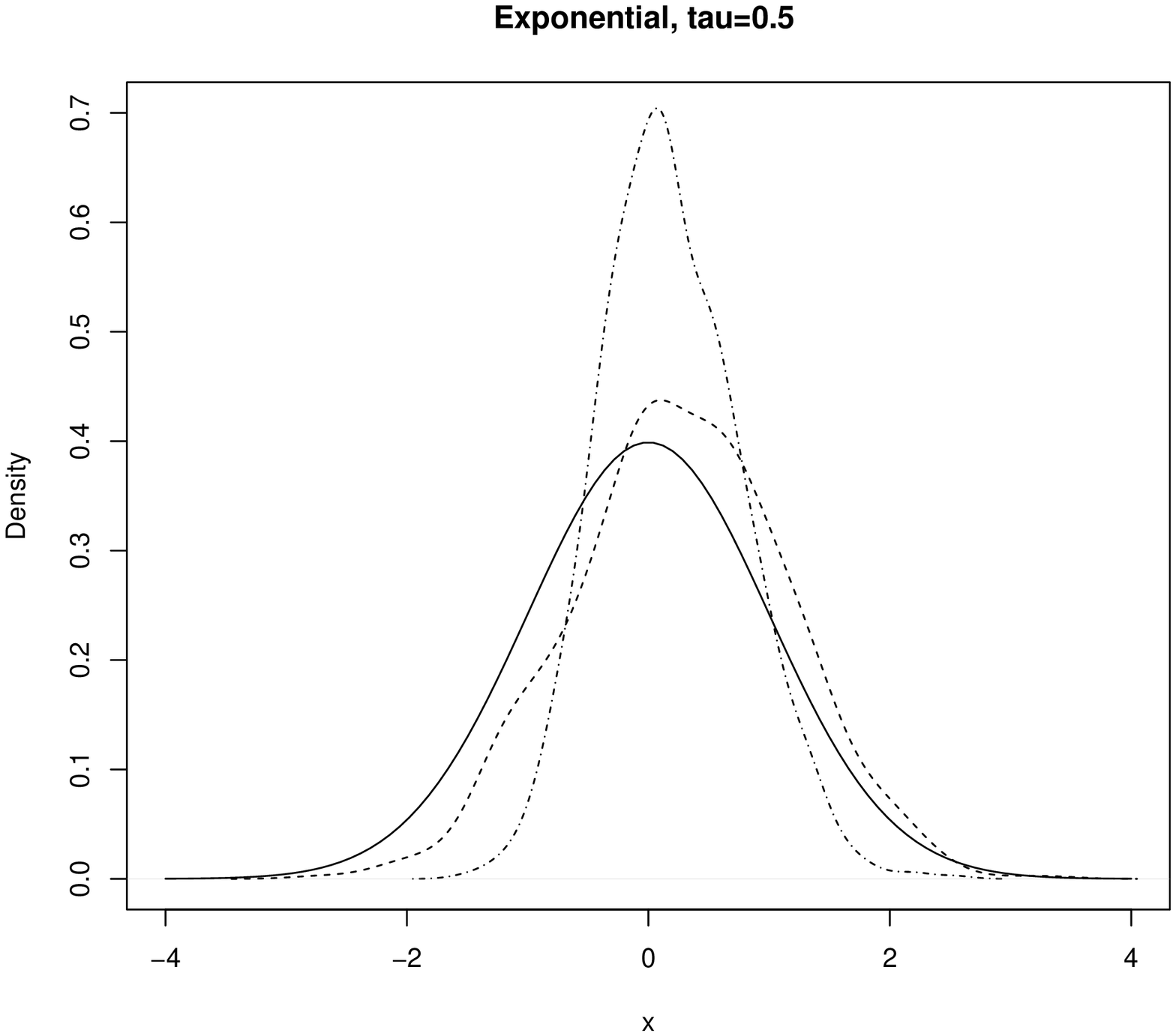}\\
\includegraphics[width=75mm,height=50mm]{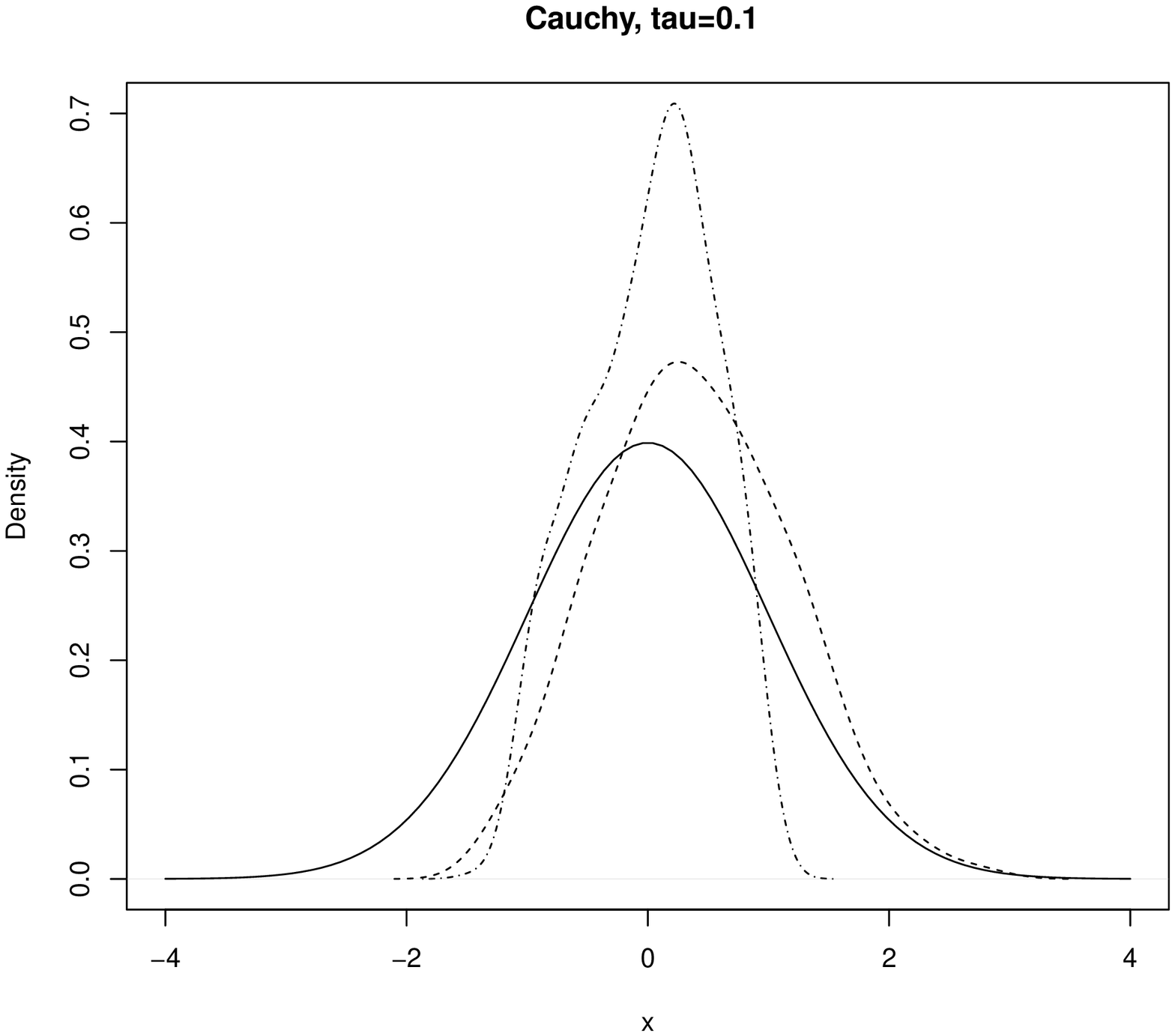}
\includegraphics[width=75mm,height=50mm]{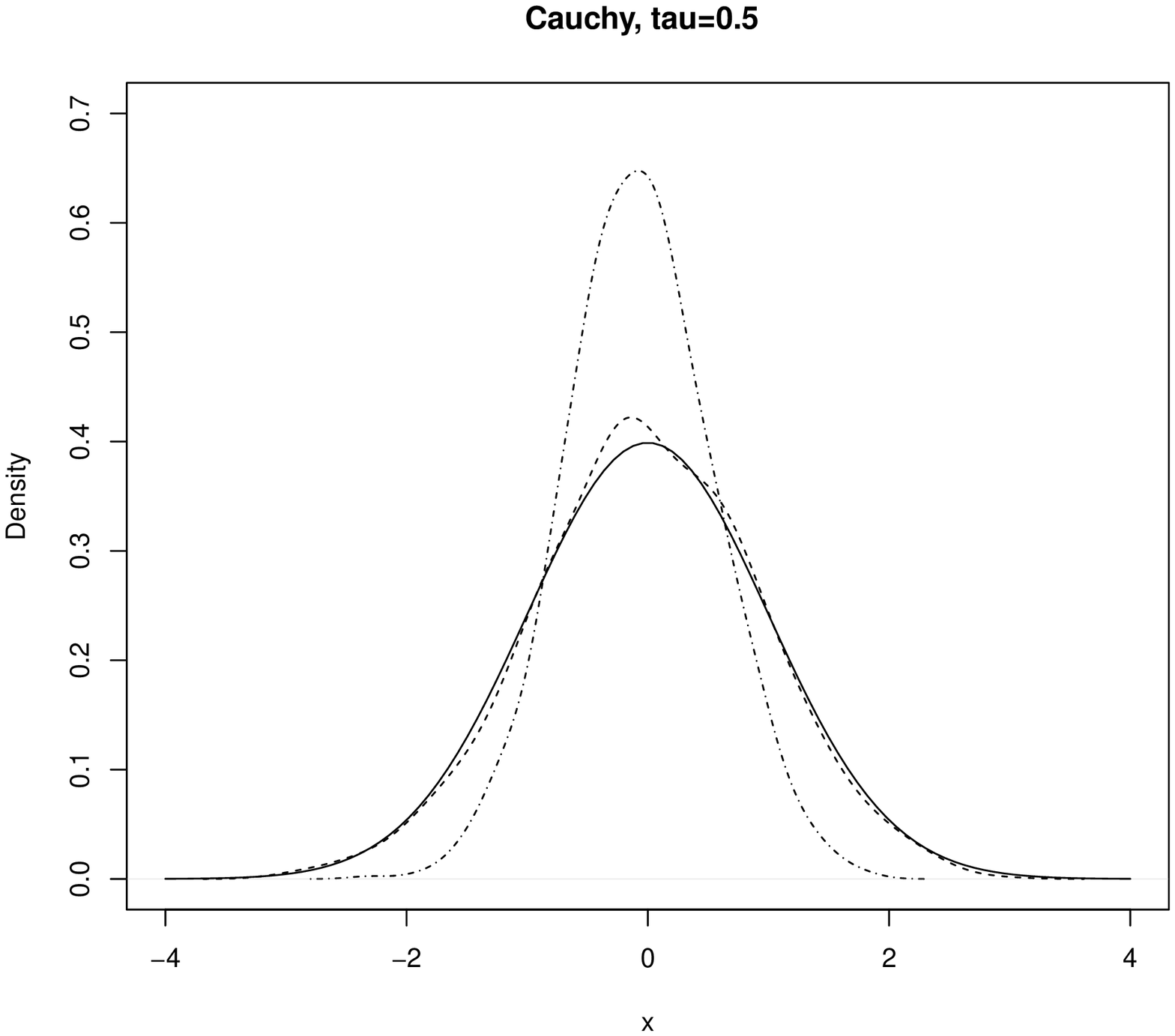}
\end{center}
\caption{The density estimate of $U_{\tau}(x)$ for $n=100$(dot-dashed) and $n=1000$(dashed), and the density of $N(0,1)$(solid). 
The left panels are for $\tau=0.1$ and the right panels are for $\tau=0.5$.
The upper, middle and bottom panels are for normal, exponential and Cauchy errors, respectively. 
  \label{simu}}
\end{figure}

\subsection{Application}

In this section, we apply the penalized spline quantile estimator to real data. 
In all examples, we use $(p,m)=(3,2)$ and $(K_n,\lambda_n)$ is chosen via GACV.

Figure \ref{bmd} showed the penalized spline quantile estimators ($\tau=0.1,\cdots,0.9$) for bone mineral density (BMD) data. 
This data was presented by Hastie et al. (2009). 
Takeuchi et al. (2006) applied the kernel estimator to the same data. 
Compared with Figure 2 (b) of their paper, the penalized splines have a somewhat smooth curve.

Next, the confidence interval of $\eta_\tau(x)$ is illustrated. 
The 100$\alpha\%$ confidence interval of $\eta_\tau(x)$ based on the asymptotic result of $\hat{\eta}_\tau(x)$ is obtained as 
\begin{eqnarray}
\left[
\hat{\eta}_\tau(x)-\hat{b}_\tau^a(x)-\hat{b}_\tau^\lambda(x)-z_{1-\alpha/2}\sqrt{\hat{\Phi}_\tau(x)},\ \ \hat{\eta}_\tau(x)-\hat{b}_\tau^a(x)-\hat{b}_\tau^\lambda(x)+z_{1-\alpha/2}\sqrt{\hat{\Phi}_\tau(x)}
\right],\label{conf}
\end{eqnarray}
where $\hat{b}_\tau^a(x)$, $\hat{b}_\tau^\lambda(x)$ and $\hat{\Phi}_\tau(x)$ are the estimators of $b_\tau^a(x)$, $b_\tau^\lambda(x)$ and $\Phi_\tau(x)$, while $z_{1-\alpha/2}$ is a $(1-\alpha/2)$th normal percentile. 
As the estimator of $b_\tau^\lambda(x)$,
\begin{eqnarray*}
\hat{b}^\lambda_\tau(x)&=&-\lambda_{n}\vec{B}(x)^T (Z^T\hat{R}Z+\lambda_{n}D_m^T D_m)^{-1}D_m^T D_m\hat{\vec{b}}(\tau)
\end{eqnarray*}
is used. 
We utilize $\hat{\Phi}_\tau(x)$ as given in the previous section.
As the pilot estimator of $\eta_\tau^{(p+1)}(x)$ in $\hat{b}_\tau^a(x)$, we construct the $(p+1)$th derivative of the penalized spline quantile estimator with the $(p+2)$th $B$-spline model. 
Thus, we obtain (\ref{conf}). 

In Figure \ref{motor}, the $95\%$ approximate confidence interval of $\eta_{0.5}(x)$ for motor cycle impact data is drawn. 
This dataset, with $\{(y_i,x_i):i=1,\cdots,132\}$ was given by H$\ddot{{\rm a}}$rdle (1990), where $y_i$ is the acceleration (g) and $x_i$ is the time (ms). 
For comparison, the $95\%$ approximate confidence interval with uncorrected bias of $\eta_{0.5}(x)$ defined by 
\begin{eqnarray*}
\left[
\hat{\eta}_\tau(x)-1.96\sqrt{\hat{\Psi}_\tau(x)},\ \ \hat{\eta}_\tau(x)+1.96\sqrt{\hat{\Psi}_\tau(x)}
\right]
\end{eqnarray*}
is shown.
The penalized spline estimator of the median has a curve with fitness and smoothness.
In the area near $x=20$, we see that there is a strong correction of the bias of $\hat{\eta}_{0.5}(x)$.

Finally, we compare the median estimator and the mean estimator for Boston housing data, with $\{(y_i,x_i):i=1,\cdots,506\}$, where $y_i$ is the median value of owner-occupied homes in USD 1000s (given by MEDV) and $x_i$ is the average number of rooms per dwelling (denoted RM). 
This dataset is available from Harrison and Rubinfeld (1979).  
Figure \ref{boston} shows the penalized spline quantile estimator of $\eta_{0.5}(x)$(solid) and the penalized spline estimator 
$$
\hat{g}(x)=\vec{B}(x)^T(Z^TZ+\mu_nD_m^TD_m)^{-1}Z^T\vec{y}
$$ 
of the conditional mean of $Y$: $g(x)=E[Y|X=x]$ (dashed), where $\mu_n$ is the smoothing parameter chosen by generalized cross-validation.  
At around $x=5$ and the right-hand side of $x=8$, the behavior of the median estimator and the mean estimator are different. 
We see that $\hat{g}(x)$ is affected by extreme points, such as $(x,y)=(4.97,50)$ and $(x,y)=(8.78,21.9)$.
On the other hand, it appears that the influence of extreme values is limited for the median estimator.

\begin{figure}
\begin{center}
\includegraphics[width=120mm,height=80mm]{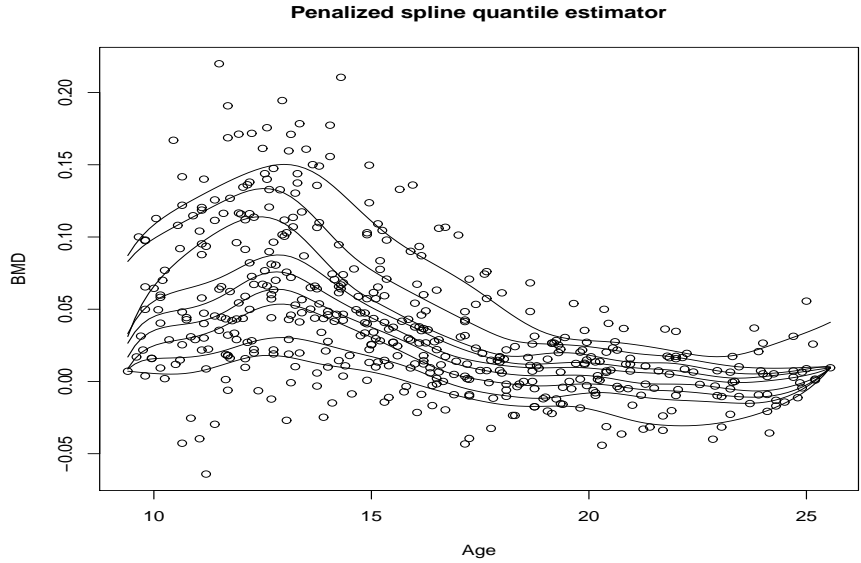}
\end{center}
\caption{BMD data ($n=485$) with $\hat{\eta}_{\tau}(x)$. The solid lines are for $\tau=$0.1, 0.2, 0.3, 0.4, 0.5, 0.6, 0.7, 0.8 and 0.9 from the bottom to top. \label{bmd}}
\end{figure}

\begin{figure}
\begin{center}
\includegraphics[width=120mm,height=80mm]{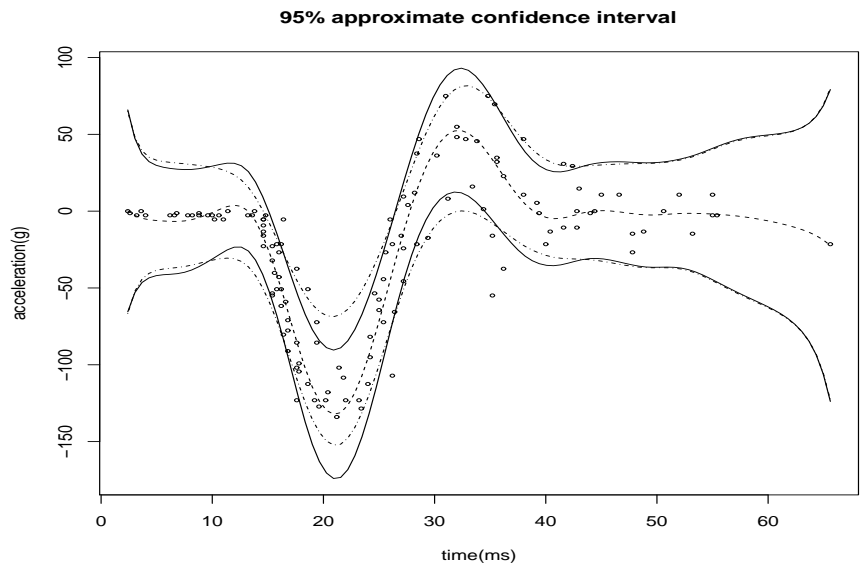}
\end{center}
\caption{Motor cycle impact data ($n=132$) with $\hat{\eta}_{0.5}(x)$ (dashed), the $95\%$ approximate confidence intervals (solid) and the $95\%$ approximate confidence intervals with uncorrected bias (dot-dashed). \label{motor}}
\end{figure}

\begin{figure}
\begin{center}
\includegraphics[width=120mm,height=80mm]{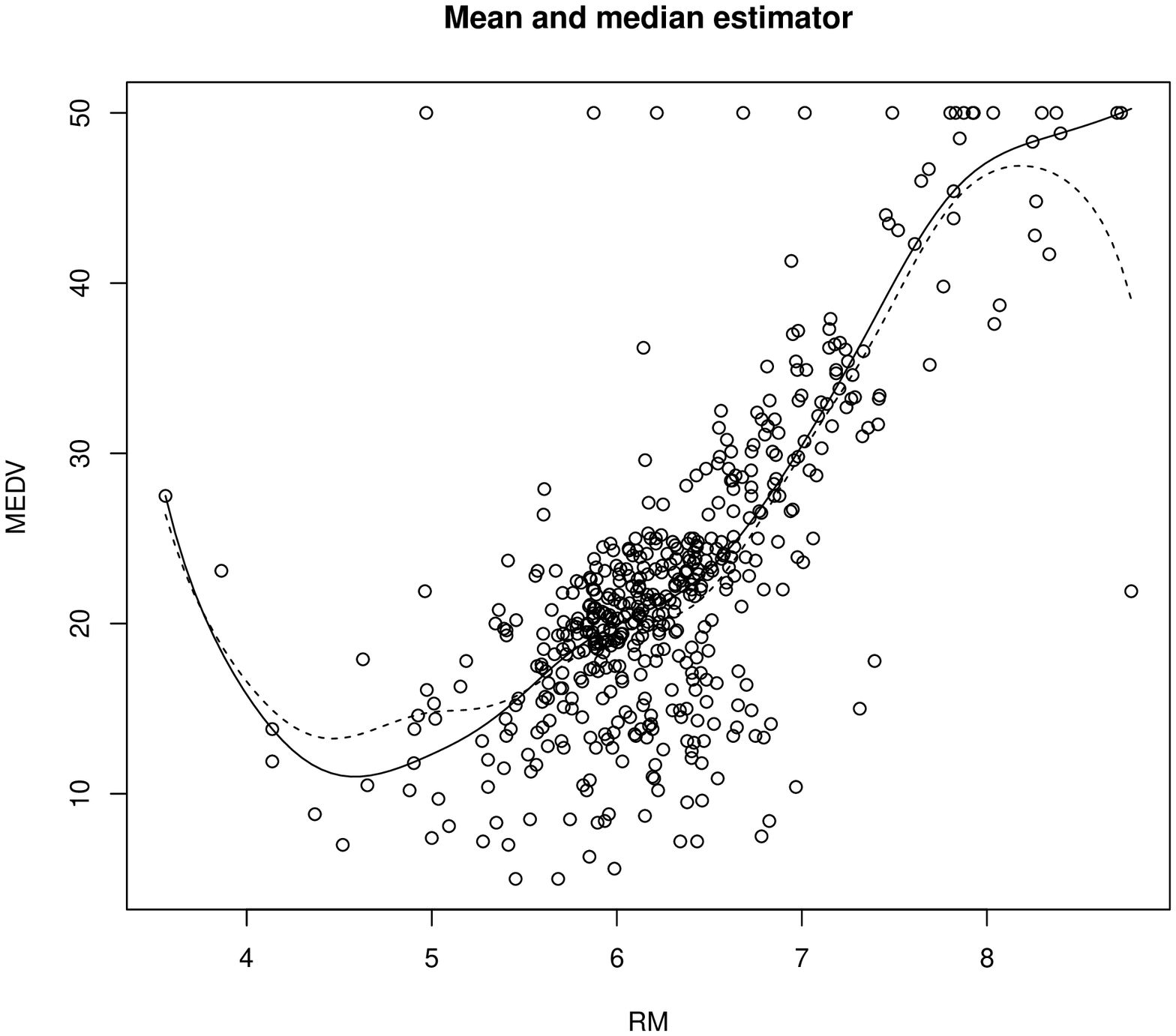}
\end{center}
\caption{Boston housing data ($n=506$) with the mean(dashed) and median(solid) estimators. \label{boston}}
\end{figure}


\section{Discussion}

This paper have discussed the asymptotic theory of the penalized spline quantile estimator. 
We showed the asymptotic bias and variance as well as the asymptotic normality of the penalized spline quantile estimator.
The results can be regarded as the quantile regression version of the Theorem 2 (a) of Claeskens et al. (2009).

As the further study, we may consider the asymptotic property of the penalized splines with multivariate covariate $(x_1,\cdots,x_d)$. 
Doskum and Koo (2000) have studied the unpenalized spline quantile estimator in additive models, but the asymptotic results were not discussed. 
The additive model has the true quantile function as 
$\eta_\tau(x_1,\cdots,x_d)=\sum_{i=1}^d \eta_{i\tau}(x_i)$. 
The aim is then to estimate $\eta_{i\tau}(x_i)$ for each $i$.
Similar to the work of Doskum and Koo, we can construct the penalized spline estimator in additive models. 
In this field, the asymptotic results should be determined.

In relation to the serious problem of the nonparametric quantile regression, a phenomenon called the ^^ ^^ quantile crossing" occurs (see Koenker (2005)). 
He (1997) and Takeuchi et al. (2006) studied the original estimation methods of $\eta_\tau(x)$ without quantile crossing. 
However, the asymptotics for their estimators have not yet been developed. 
The asymptotic study of the penalized splines without quantile crossing would be an interesting topic for further study. 
In addition, by using the asymptotic results of the penalized spline estimator $\hat{\eta}_\tau(x)$, it may be possible to construct the penalized spline quantile estimator without quantile crossing although this is beyond the scope of this paper.

\section*{Appendix}

For a random variable $U_n$, $E[U_n|\vec{X}_n]$ and $V[U_n|\vec{X}_n]$ denote the conditional expectation and variance of $U_n$ given $(\vec{X}_1,\cdots,\vec{X}_n)=(\vec{x}_1,\cdots,\vec{x}_n)$, respectively.
For the matrix $A=(a_{ij})_{ij}$, $||A||_{\infty}=\max_{ij}\{|a_{ij}|\}$.
For random sequence $\{a_n\}$ and $\{b_n\}$, if $a_n/b_n=O_P(1)$, then it is written as $a_n\stackrel{as}{\sim} b_n$. 

\begin{lemma}\label{G1}
Let $A=(a_{ij})_{ij}$ be $(K_n+p)$ matrix and let $H(\tau)=G(\tau)+(\lambda_{n}/n)D_m^TD_m$. 
Assume that $K_n\rightarrow \infty$ as $n\rightarrow \infty$, $||A||_{\infty}=O_P(K_n^\alpha)$. 
Then, under the Assumption, $||AG||_{\infty}=O(K_n^{\alpha-1})$ and $||AH(\tau)^{-1}||_{\infty}=O(K_n^{1+\alpha})$.
\end{lemma}

Lemma \ref{G1} can be proven similar to Lemma 1 of Claeskens et al. (2009). 
Then, Assumption 5 which guarantees $K_q<1$ that given in their paper. 

\begin{lemma}\label{Lyapnov}
Let $\psi_\tau(u)=\tau-I(u<0)$,  $u_i=y_i-\vec{B}(x_i)^T\vec{b}^*(\tau)$.
Under the same assumption as Theorem \ref{clt}, 
\begin{eqnarray*}
-\sqrt{\frac{K_n}{n}}\sum_{i=1}^n \vec{B}(x_i)^T \vec{\delta}\psi_\tau(u_i) \stackrel{as}{\sim}  \sqrt{K_n}W^T\vec{\delta},
\end{eqnarray*}
where $W\sim N(\vec{0},\tau(1-\tau)G)$.
\end{lemma}

\begin{proof}[Proof of Lemma \ref{Lyapnov}]

Let 
$$
Z_{n}=-\sqrt{\frac{K_n}{n}}\sum_{i=1}^n \vec{B}(x_i)^T \vec{\delta}\psi_\tau(u_i).
$$
We show the asymptotic distribution of $Z_n$ by Lyapunov's theorem.
First from the fact that $P(Y<\eta_\tau(x_i)|X_i=x_i)=\tau$, we have 
\begin{eqnarray*}
E[\psi_\tau(U_i)|\vec{X}_n]
&=&
\tau-E[I(Y_i<\vec{B}(x_i)^T\vec{b}^*(\tau))|\vec{X}_n]\\
&=&
\tau-P(Y<\vec{B}(x_i)^T\vec{b}^*(\tau)|X_i=x_i)\\
&=&
\tau-P(Y<\eta_\tau(x_i)+b_a(x_i,\tau)(1+o(1))|X_i=x_i)\\
&=&
-b_a(x_i,\tau)f(\eta_\tau(x_i)|x_i)(1+o(1))\\
&=&
o(1).
\end{eqnarray*}
Therefore we obtain 
\begin{eqnarray*}
&&E\left[\left.\left|\sqrt{\frac{K_n}{n}}\vec{B}(x_i)^T \vec{\delta}\{\psi_\tau(U_i)-E[\psi_\tau(U_i)|\vec{X}_n]\}\right|^{2+\gamma}\right|\vec{X}_n\right]\\
&&=\left(\frac{K_n}{n}\right)^{(2+\gamma)/2}|\vec{B}(x_i)^T \vec{\delta}|^{2+\gamma}E[|\psi_\tau(U_i)|^{2+\gamma}+o(1)|\vec{X}_n]\\
&&\leq O\left(\left(\frac{K_n}{n}\right)^{(2+\gamma)/2}\right).
\end{eqnarray*}
The straightforward calculation yields 
\begin{eqnarray*}
V[Z_n|\vec{X}_n]
&=&
\frac{K_n}{n}\sum_{i=1}^n \{\vec{B}(x_i)^T \vec{\delta}\}^2V[\psi_\tau(Y_i-\vec{B}(x_i)^T\vec{b}^*(\tau))|\vec{X}_n]\\
&=&
\tau(1-\tau)\vec{\delta}^T\left(\frac{K_n}{n}\sum_{i=1}^n \vec{B}(x_i)\vec{B}(x_i)^T \right)\vec{\delta}\\
&=&
K_n\tau(1-\tau)\vec{\delta}^T G\vec{\delta}(1+o_P(1))\\
&=&O(K_n)
\end{eqnarray*}
So it follows that 
\begin{eqnarray*}
&&\frac{1}{V[Z_n|\vec{X}_n]^{(2+\gamma)/2}}\sum_{i=1}^n E\left[\left.\left|\sqrt{\frac{K_n}{n}}\vec{B}(x_i)^T \vec{\delta}\{\psi_\tau(U_i)-E[\psi_\tau(U_i)|\vec{X}_n]\}\right|^{2+\gamma}\right|\vec{X}_n\right]\\
&&\leq O(K_n^{-(2+\gamma)/2})O\left(n\left(\frac{K_n}{n}\right)^{(2+\gamma)/2}\right)\\
&&=o(1)
\end{eqnarray*}
since $\gamma\geq 0$.
This leads to  
\begin{eqnarray*}
\frac{Z_n-E[Z_n|\vec{X}_n]}{V[Z_n|\vec{X}_n]}\stackrel{D}{\longrightarrow} N(0,1)
\end{eqnarray*} 
from Lyapnov's theorem.
The expectation of $Z_{n}$ can be calculated as 
\begin{eqnarray*}
E[Z_n|\vec{X}_n]
&=&
-\sqrt{\frac{K_n}{n}}\sum_{i=1}^n \vec{B}(x_i)^T \vec{\delta}E[\psi_\tau(U_i)|\vec{X}_n]\\
&=&
\sqrt{\frac{K_n}{n}}\sum_{i=1}^n \vec{B}(x_i)^T \vec{\delta}b_a(x_i,\tau)f(\eta_\tau(x_i)|x_i)(1+o_P(1))\\
&=&
\sqrt{nK_n}\int_0^1 \vec{B}(u)^T \vec{\delta}b_a(u,\tau)f(\eta_\tau(u)|u)du(1+o_P(1)).
\end{eqnarray*}
From the proof of Lemma 6.10 of Argwall and Studen (1989), for $j=-p+1,\cdots,K_n$, we have
\begin{eqnarray*}
\int_0^1 B_j(u)b_a(u,\tau)f(\eta_\tau(u)|u)du(1+o(1))=o(K_n^{-(p+2)}),
\end{eqnarray*}  
by which $\sqrt{nK_n}o(K_n^{-(p+2)})=o(1)$.
Consequently, we have $E[Z_n|\vec{X}_n]/V[Z_n|\vec{X}_n]=o_P(1)$ and Lemma \ref{Lyapnov} holds.
\end{proof}

\begin{lemma}\label{var}
Let $w_{in}=\sqrt{K_n/n}\vec{B}(x_i)^T \vec{\delta} (i=1,\cdots,n)$ for $\vec{\delta}\in\mathbb{R}^{K_n+p}$. 
Then, under the assumptions, 
\begin{eqnarray*}
\sum_{i=1}^n \int_0^{w_{in}} \{I(u_i\leq s)-I(u_i\leq 0)\}ds
\stackrel{as}{\sim} \frac{K_n}{2}\vec{\delta}^T G(\tau)\vec{\delta}.
\end{eqnarray*}
\end{lemma}

\begin{proof}[Proof of Lemma \ref{var}]
Let 
$$
R_n=\sum_{i=1}^n \int_0^{w_{in}} \{I(u_i\leq s)-I(u_i\leq 0)\}ds.
$$
Since 
\begin{eqnarray*}
&&E\left[\left.\int_0^{w_{in}} \{I(u_i\leq s)-I(u_i\leq 0)\}ds\right|\vec{X}_n\right]\\
&&=
\int_0^{w_{in}} E[\{I(U_i\leq s)-I(U_i\leq 0)\}|\vec{X}_n]ds\\
&&=
\int_0^{w_{in}} \left\{P\left(Y_i<\vec{B}(x_i)^T\vec{b}^*(\tau)+s|X_i=x_i\right)-P(Y_i<\vec{B}(x_i)^T\vec{b}^*(\tau)|X_i=x_i)\right\}ds\\
&&=
{\small
\sqrt{\frac{K_n}{n}}\int_0^{\vec{B}(x_i)^T\vec{\delta}} \left\{P\left(\left.Y_i<\vec{B}(x_i)^T\vec{b}^*(\tau)+t\frac{K_n}{n}\right|X_i=x_i\right)-P(Y_i<\vec{B}(x_i)^T\vec{b}^*(\tau)|X_i=x_i)\right\}dt
}\\
&&=
\frac{K_n}{n}\int_0^{\vec{B}(x_i)^T\vec{\delta}} f\left(\vec{B}(x_i)^T\vec{b}^*(\tau)|x_i\right)tdt\\
&&=
\frac{K_n}{2n}f\left(\vec{B}(x_i)^T\vec{b}^*(\tau)|x_i\right)\{\vec{B}(x_i)^T\vec{\delta}\}^2.
\end{eqnarray*}
Therefore we obtain
\begin{eqnarray*}
E[R_n|\vec{X}_n]
&=&
\frac{K_n}{2n}\sum_{i=1}^n f\left(\vec{B}(x_i)^T\vec{b}^*(\tau)|x_i\right)\vec{\delta}^T\vec{B}(x_i)\vec{B}(x_i)^T\vec{\delta}\\
&=&
\frac{K_n}{2}\vec{\delta}^T \left(\frac{1}{n}\sum_{i=1}^n f\left(\eta_\tau(x_i)+o(1)|x_i\right)\vec{B}(x_i)\vec{B}(x_i)^T\right)\vec{\delta}\\
&=&
\frac{K_n}{2}\vec{\delta}^T G(\tau)\vec{\delta}(1+o_P(1)).
\end{eqnarray*}
Finally, we show $V[R_n|\vec{X}_n]=o_P(1)$. 
For $i=1,\cdots,n$, we have 
\begin{eqnarray*}
\int_0^{w_{in}} \{I(u_i\leq s)-I(u_i\leq 0)\}ds\leq \sqrt{\frac{K_n}{n}}\vec{B}(x_i)^T \vec{\delta}. 
\end{eqnarray*} 
Therefore the variance of $R_n$ can be evaluated as 
\begin{eqnarray*}
V[R_n|\vec{X}_n]
&\leq &
\sum_{i=1}^n E\left[\left.\left(\int_0^{w_{in}} \{I(u_i\leq s)-I(u_i\leq 0)\}ds\right)^2\right|\vec{X}_n\right]\\
&\leq & 
\sqrt{\frac{K_n}{n}}\max_{i=1,\cdots,n}\{\vec{B}(x_i)^T \vec{\delta}\}E[R_n|\vec{X}_n].
\end{eqnarray*}
Since $E[R_n|\vec{X}_n]=O(K_n)$, we obtain $\sqrt{V[R_n|\vec{X}_n]}/E[R_n|\vec{X}_n]=o_P(1)$ and, hence, Lemma \ref{var} holds.

\end{proof}

\begin{proof}[Proof of Proposition \ref{para}]

Let 
\begin{eqnarray*}
U_n(\vec{\delta})&=&\sum_{i=1}^n \left[\rho_\tau\left(u_i-\sqrt{\frac{K_n}{n}}\vec{B}(x_i)^T \vec{\delta}\right) -\rho_\tau(u_i)\right]\\
&&+\frac{\lambda_{n}}{2}\left(\vec{b}^*(\tau)+\sqrt{\frac{K_n}{n}}\vec{\delta}\right)^T D_m^T D_m\left(\vec{b}^*(\tau)+\sqrt{\frac{K_n}{n}}\vec{\delta}\right)-\frac{\lambda_{n}}{2}\vec{b}^*(\tau)^TD_m^T D_m\vec{b}^*(\tau),
\end{eqnarray*}
where $u_i=y_i-\vec{B}(x_i)^T\vec{b}^*(\tau)$.
Then the minimizer $\hat{\vec{\delta}}_n(\tau)$ of $U_n(\delta)$ can be obtained as 
\begin{eqnarray*}
\hat{\vec{\delta}}_n(\tau)=\sqrt{\frac{n}{K_n}}(\hat{\vec{b}}(\tau)-\vec{b}^*(\tau)).
\end{eqnarray*}
First we show the convergence point $U_0(\vec{\delta})$ of $U_n(\vec{\delta})$ for any $\vec{\delta}\in\mathbb{R}^{K_n+p}$. 
For the following discussion, we introduce the Knight's idntity(see, Knight (1998)): 
\begin{eqnarray}
\rho_\tau(u-v)-\rho_\tau(u)=-v\psi_\tau(u)+\int_0^v \{I(u\leq s)-I(u\leq 0)\}ds,\label{knight}
\end{eqnarray}
where $\psi_\tau(u)=\tau-I(u<0)$. 
By using (\ref{knight}), we can write $U_n(\delta)$ as 
\begin{eqnarray*}
U_n(\delta)=U_{1n}(\vec{\delta})+U_{2n}(\vec{\delta})+U_{3n}(\vec{\delta})+U_{4n}(\vec{\delta}),
\end{eqnarray*}
where 
\begin{eqnarray*}
U_{1n}(\vec{\delta})
&=&
-\sqrt{\frac{K_n}{n}}\sum_{i=1}^n \vec{B}(x_i)^T \vec{\delta}\psi_\tau(u_i),\\
U_{2n}(\vec{\delta})
&=&
\sum_{i=1}^n \int_0^{w_{in}} \{I(u_i\leq s)-I(u_i\leq 0)\}ds,\\
U_{3n}(\vec{\delta})
&=&\frac{\lambda_{n}K_n}{2n}\vec{\delta}^TD_m^T D_m\vec{\delta},\\
U_{4n}(\vec{\delta})
&=&
\lambda_{n}\sqrt{\frac{K_n}{n}}\vec{b}^*(\tau)^TD_m^T D_m\vec{\delta},
\end{eqnarray*}
where $w_{in}=\sqrt{K_n/n}\vec{B}(x_i)^T \vec{\delta}$. 
From Lemma 1, $U_{1n}(\vec{\delta})$ satisfies 
\begin{eqnarray*}
U_{1n}(\vec{\delta})\stackrel{as}{\sim} -\sqrt{K_n}W^T\vec{\delta},
\end{eqnarray*}
where $W\sim N(\vec{0},\tau(1-\tau)G)$. 
Furthermore Lemma 2 and $U_{3n}(\vec{\delta})$ yield 
\begin{eqnarray*}
U_{2n}(\vec{\delta})+U_{3n}(\vec{\delta})\stackrel{as}{\sim} \frac{K_n}{2}\vec{\delta}^T \left(G(\tau)+\frac{\lambda_{n}}{n}D_m^T D_m\right)\vec{\delta}.
\end{eqnarray*}
Therefore, we obtain  
\begin{eqnarray*}
U_n(\vec{\delta})\stackrel{as}{\sim} U_0(\vec{\delta})=-\sqrt{K_n}W^T\vec{\delta}+\lambda_{n}\sqrt{\frac{K_n}{n}}\vec{b}^*(\tau)^T D_m^T D_m\vec{\delta} +\frac{K_n}{2}\vec{\delta}^T \left(G(\tau)+\frac{\lambda_{n}}{n}D_m^T D_m\right)\vec{\delta}.
\end{eqnarray*}
Because $U_0(\vec{\delta})$ is convex with respect to $\vec{\delta}$ and has unique minimizer, the minimizer $\hat{\vec{\delta}}_n(\tau)$ of $U_n(\vec{\delta})$ converge to $\vec{\delta}_0(\tau)=\argmin_{\vec{\delta}}\{U_0(\vec{\delta})\}$.
This fact is detailed in Pollard(1991), Knight (1998) and Kato (2009). 
Hence we have 
\begin{eqnarray*}
\sqrt{\frac{n}{K_n}}\{\hat{\vec{b}}(\tau)-\vec{b}^*(\tau)\}\stackrel{as}{\sim} \ \vec{\delta}_0(\tau)=\left(G(\tau)+\frac{\lambda_{n}}{n}D_m^T D_m\right)^{-1}\left(\frac{1}{\sqrt{K_n}}W-\frac{\lambda_{n}}{\sqrt{nK_n}} D_m^T D_m\vec{b}^*(\tau)\right).
\end{eqnarray*}  
Since $\hat{\eta}_\tau(x)-\eta^*_{\tau}(x)=\vec{B}(x)^T(\hat{\vec{b}}(\tau)-\vec{b}^*(\tau))$, we obtain for $x\in(0,1)$, as $n\rightarrow \infty$,
\begin{eqnarray*}
\sqrt{\frac{n}{K_n}}\{\hat{\eta}_\tau(x)-\eta^*_{\tau}(x)-b^\lambda_\tau(x)\}\stackrel{D}{\longrightarrow} N(0,\Phi_\tau(x))
\end{eqnarray*}
by the definition of $W$.
We can confirm with Lemma \ref{G1} that $\Phi_\tau(x)=O(1)$. 
Finally we show the asymptotic order of $b_\tau^{\lambda}(x)$. 
Let $\vec{B}^{[p]}(x)=(B_{-p+1}^{[p]}(x)\ \cdots\ B_{K_n}^{[p]}(x))^T$.
By the properties of the derivative of the $B$-spline model, we have $s_\tau^{(m)}(x)=\partial^m s_\tau(x)/\partial x^m=K_n^m \vec{B}^{[p-m]}(x)^T D_m\vec{b}(\tau)$. 
Therefore we obtain $\vec{B}^{[p-m]}(x)^T \{K_n^mD_m\vec{b}^*(\tau)\}=\eta_\tau^{(m)}(x)(1+o(1))$ for $m\leq p$. 
Since the asymptotic order of $\vec{B}^{[p-m]}(x)^T \{K_n^mD_m\vec{b}^*(\tau)\}$ and that of $||K_n^mD_m\vec{b}^*(\tau)||_{\infty}$ are the same as $O(1)$, $||D_m\vec{b}^*(\tau)||_{\infty}=O(K_n^{-m})$ is satisfied for $m\leq p$. 
In addition, similar to the proof of Theorem 1 of Kauermann et al. (2009), $||D_{p+1}\vec{b}^*(\tau)||_{\infty}=O(K_n^{-(p+1)})$ is fulfilled. 
Together with Lemma \ref{G1}, we obtain 
\begin{eqnarray*}
b_\tau^{\lambda}(x)=-\frac{\lambda_{n}}{n}\vec{B}(x)^T\left(G(\tau)+\frac{\lambda_{n}}{n}D_m^T D_m\right)^{-1}D_m^T D_m\vec{b}^*(\tau)=O(\lambda_{n}n^{-1}K_n^{1-m})=O(n^{-(p+1)/(2p+3)}).
\end{eqnarray*}
Thus Proposition 2 has been proven.
\end{proof}

\begin{proof}[Proof of Theorem \ref{clt}]
Theorem \ref{clt} can be proven directly from Propositions 1. 
Under the condition $K_n=O(n^{1/(2p+3)})$, we have
\begin{eqnarray*}
\sqrt{\frac{n}{K_n}}\{\hat{\eta}_\tau(x)-\eta^*_{\tau}(x)-b^\lambda_\tau(x)\}=\sqrt{\frac{n}{K_n}}\{\hat{\eta}_\tau(x)-\eta^*_{\tau}(x)-b^a_\tau(x)+o(K_n^{-(p+1)})-b^\lambda_\tau(x)\}
\end{eqnarray*}
and $\sqrt{n/K_n}b^a_\tau(x)=O(\sqrt{n/K_n}K_n^{-(p+1)})=O(1)$. 
This completes the proof.
\end{proof}

\def\bibname{Reference}


\begin{thebibliography}{10}
\bibitem{A}Agawal,G. and Studden,W.(1980), ``Asymptotic integrated mean square error using least squares and bias minimizing splines,''{\it Ann. Statist.} {\bf 8},1307-1325.
\bibitem{C}Claeskens,G., Krivobokova,T. and Opsomer,J.D.(2009). Asymptotic properties of penalized spline estimators.\ $Biometrika.$ $\mathbf{96}$,
529-544.
\bibitem{d}de Boor,C.(2001). $A\ Practical\ Guide\ to\ Splines$. Springer-Verlag.
\bibitem{E}Eilers,P.H.C. and Marx,B.D.(1996). Flexible smoothing with $B$-splines and penalties(with Discussion). $Statist.Sci$. {\bf 11}, 89-121.
\bibitem{F}Fan,J., Hu,T.C., and Truong,Y.K.(1994). Robust Nonparametric Function Estimation. {\it Scandinavian Journal of Statistics}. {\bf 21}, 433-446.
\bibitem{H}Hao,L. and Naiman,D.Q.(2007). {\it Quantile regression}. Sage Publications, Inc.
\bibitem{H}Hastie,T., Tibshirani,R. and Friedman,J.(2009). {\it The Elements of Statistical Learning},\ Springer-Verlag.
\bibitem{H}He,X. and Shi,P.(1994). Convergence rate of B-spline estimators of nonparametric conditional quantile functions. {\it J. Nonparam. Statist}. {\bf 3}, 299-308.
\bibitem{He}Hendricks,W. and Koenker,R.(1992). Hierarchical spline models for conditional quantiles and the demand for electricity. {\it J. Amer. Statist. Assoc.} {\bf 87}, 58-68.
\bibitem{K}Kai,B., Li,R., and Zou,H.(2011). New efficient estimation and variable selection methods for semiparametric varying-coefficient partially linear models. {\it Ann. Statist.} {\bf 39}, 305-332.
\bibitem{K}Kato,K.(2009). Asymptotics for argmin processes: convexity arguments. {\it J. Multi. Anal}. {\bf 100}, 1816-1829.
\bibitem{K}Kauermann,G., Krivobokova,T., and Fahrmeir,L.(2009). Some asymptotic results on generalized penalized spline smoothing.{\it J. R. Statist. Soc.} B {\bf 71}, 487-503.
\bibitem{K}Knight.K.(1998). Limiting distributions for $L_1$ regression estimators under general conditions. {\it Ann. Statist.} {\bf 26}, 755-770.
\bibitem{Ko}Koenker,R. (2005). {\it Quantile regression}. Cambridge Univ. Press, Cambridge.
\bibitem{Ko}Koenker,R. and Bassett,G.(1978). Regression quantiles. {\it Econometrica}. {\bf 46}, 33-50.
\bibitem{Ko}Koenker,R., Ng,P. and Portnoy,S.(1994). Quantile smoothing splines. {\it Biometrika}. {\bf 81}, 673-680.
\bibitem{K}Koenker,R. and Park,B.J.(1996). An interior point algorithm for nonlinear quantile regression. {\it J. Econom}. {\bf 71}, 265-283.
\bibitem{N}Nychka,D., Gray,G., Haaland,P., Martin,D., and O'Connell,M.(1995). A nonparametric regressio	n approach to syringe grading for quality improvement. {\it J. Amer. Statist. Assoc.} {\bf 90}, 1171-1178.
\bibitem{P}Pollard,D.(1991). Asymptotics for least absolute deviation regression estimators. {\it Econometric Theory}. {\bf 7}, 186-199.
\bibitem{P}Portnoy,S.(1997). Local asymptotics for quantile smoothing splines. {\it Ann. Statist.} {\bf 25}, 414-434.
\bibitem{P}Pratesi,M., Ranalli.M.G., and Salvati,N.(2009). Nonparametric M-quantile regression using penalised splines. {\it J. Nonparam. Statist}. {\bf 21}, 287-304.
\bibitem{R}Reiss,P.T. and Huang,L.(2012). Smoothness selection for penalized quantile regression splines. {\it The International Journal of Biostatistics}. {\bf 8.1}. 
\bibitem{S}Sheather, S. J. and Jones, M. C.(1991). A reliable data-based bandwidth selection method for kernel density estimation. {\it J. R. Statist. Soc.} 53, 683-690. 
\bibitem{S}Shi,P. and Li,G.(1995). Global convergence rates of $B$-spline $M$-estimators in nonparametric regression. {\it Statistica Sinica}. {\bf 5}, 303-318.
\bibitem{T}Takeuchi,I., Li,Q.V, Sears,T.D., and Smola,A.J.(2006). Nonparametric quantile estimation. {\it Journal of Machine Learning Research}. {\bf 7}, 1231-1264.
\bibitem{Y}Yu,K. and Jones,M.C.(1998). Local linear quantile regression. {\it J. Amer. Statist. Assoc.} {\bf 93}, 228-237.
\bibitem{Y}Yuan,M.(2006). GACV for quantile smoothing splines.{\it Computational Statistics $\&$ Data Analysis}. {\bf 50}, 813-829.
\bibitem{Z}Zhou,S., Shen,X. and Wolfe,D.A.(1998). Local asymptotics for regression splines and confidence regions. {\it Ann. Statist.} {\bf 26}(5):1760-1782.
\end{thebibliography}
\end{document}